\documentclass[letterpaper]{amsart}
\usepackage[utf8]{inputenc}
\pdfoutput=1
\usepackage{amsmath}
\usepackage{amssymb}
\usepackage{amsthm}

\usepackage[T1]{fontenc}
\usepackage{lmodern}

\usepackage{mathrsfs}
\usepackage{enumerate}
\usepackage[labelfont=bf]{caption}
\usepackage{accents}
\usepackage{microtype}
\usepackage{mathtools}

\usepackage{hyperref}
\usepackage{amsmath}
\usepackage{amsthm} 
\usepackage{amssymb}
\usepackage{xfrac}
\usepackage{mathtools}
\usepackage{tikz-cd}
\usepackage[autostyle]{csquotes}
\usepackage{enumitem}
\usepackage{comment}
\usepackage{scalerel}
\usepackage{stackengine,wasysym}

\usepackage{todonotes}

\DeclarePairedDelimiter\abs{\lvert}{\rvert}%

\DeclareMathOperator{\supp}{supp}

\newcommand{\Isom}[1]{\mathrm{Is}(#1)}
\newcommand{\Sub}[1]{\mathrm{Sub}(#1)}
\newcommand{\DSub}[1]{\mathrm{DSub}(#1)}
\newcommand{\CAT}[1]{\mathrm{CAT}(#1)}
\newcommand{\Cl}[1]{\mathrm{Cl}(#1)}

\theoremstyle{plain}
\newtheorem{thm}{Theorem}[section]
\newtheorem{theorem}[thm]{Theorem}

\newtheorem{lemma}[thm]{Lemma}
\newtheorem{remark}[thm]{Remark}
\newtheorem{example}[thm]{Example}
\newtheorem{prop}[thm]{Proposition}
\newtheorem{cor}[thm]{Corollary}
\newtheorem*{thm*}{Theorem}

\numberwithin{equation}{section}
\theoremstyle{definition}
\newtheorem{defn}[thm]{Definition}
\newtheorem{notation}[thm]{Notation}

\newtheorem*{remarkkk}{Remark}

\title{Stationary random subgroups in negative curvature}
\author{Ilya Gekhtman and Arie Levit}

\begin{document}

\begin{abstract}
We show that discrete stationary random subgroups of isometry groups of Gromov hyperbolic spaces have full limit sets as well as critical exponents bounded from below. This information is used to answer a question of Gelander and show that a rank one locally symmetric space for which the bottom of the spectrum of the Laplace--Beltrami operator is the same as that of its universal cover has unbounded injectivity radius.
\end{abstract}

\maketitle

\section{Introduction}
\label{sec:intro}

Let $G$ be a connected center-free  simple rank one Lie group 
with associated  symmetric space $X$. 
In other words $X$ is a real, complex or quaternionic hyperbolic space or the Cayley plane and its group of isometries is $G =  \Isom{X}$.

A discrete subgroup $\Gamma \le G$ is called \emph{confined} if there exists a compact subset $Q \subset G$ such that $\Gamma^g \cap Q \neq \{e\}$ for all elements $g \in G$. If the discrete subgroup  $\Gamma$ is torsion-free then being confined is equivalent to the geometric condition saying that the locally symmetric  manifold $M_\Gamma = \Gamma \backslash X$ has an upper bound on its injectivity radius at all points. This says that the subgroup $\Gamma$ is qualitatively \enquote{large}. 

%Qualitatively this can be seen as saying that $\Gamma$ is \enquote{large}.

The \emph{critical exponent} of a discrete subgroup $\Gamma \le \Isom{X}$ is a quantitative measure of the exponential growth rate of $\Gamma$-orbits in the space $X$. 
This is a real number $\delta(\Gamma)$ taking values in the range $\left[0,d(X)\right]$ where $d(X)=\dim_{\mathrm{Haus}}(\partial X)$ is the Hausdorff dimension of the visual  boundary at infinity. For instance $d(\mathbb{H}^n) = n-1$ for the real hyperbolic space $\mathbb{H}^n$ of dimension $n$. See \S\ref{sec:critical exponents} for the precise definitions.
%The critical exponent of a lattice subgroup is equal to $d(X)$.

Our main result relates these  qualitative and  quantitative notions.

% Indeed, for any subset $A\subset Isom(X)$ we define 
% $\delta(A)=\lim \sup_{R\to \infty}\frac{1}{R}\ln \|\Gamma o\cap B_{R}(o)\|$ where $o\in X$ is any basepoint and $B_R(o)$ is the ball of radius $R$ in $(X,d)$ around $o$.
% This is a real number $\delta(\Gamma)$ taking values in the range $\left[0,d(X)\right]$ where $d(X)=\dim_{\mathrm{Haus}}(\partial X)$ is the Hausdorff dimension of the   boundary at infinity. For instance $d(\mathbb{H}^n) = n-1$ for the real hyperbolic space $\mathbb{H}^n$ of dimension $n$.  If $\Gamma$ is a lattice in $G$, i.e. if  $G/\Gamma$ has cofinite Haar measure, then $\delta(\Gamma)=d(X)$.
% Our main result relates these two qualitative and  quantitative notions.

\begin{theorem}
\label{thm:not confined symmetric spaces}
Let  $\Gamma$ be  a discrete subgroup of the simple rank one Lie group $G$. If the subgroup $\Gamma$ satisfies $\delta(\Gamma) \le \frac{1}{2}d(X)$ then  $\Gamma$ is not confined.
\end{theorem}

 This can be seen as  a conditional \emph{rank one} analogue of the recent breakthrough by Fraczyk and Gelander \cite{fraczyk2023infinite} who showed that any infinite covolume discrete subgroup of a \emph{higher rank} simple Lie group is not confined. Indeed Theorem \ref{thm:not confined symmetric spaces} has been conjectured and introduced to us  by Gelander.

Assuming the discrete subgroup $\Gamma$ is torsion-free  its critical exponent $\delta(\Gamma)$ is  related to  the \emph{bottom of the spectrum} $\lambda_0(M_\Gamma)$ of the Laplace--Beltrami operator on the locally symmetric manifold $M_\Gamma$. The Elstrodt  formula \cite{elstrodt1974resolvente,patterson1976limit,sullivan1987related,corlette1990hausdorff} states that
\begin{equation}
\lambda_0(M_\Gamma) = \begin{cases} 
 \frac{1}{4}d(X) ^2 & \text{if} \quad  \delta(\Gamma) \in \left[0,   \frac{1}{2}d(X) \right], \\ 
\delta(\Gamma)(d(X) - \delta(\Gamma)) & \text{if}  \quad \delta(\Gamma) \in \left[\frac{1}{2}d(X),   d(X) \right]. \end{cases}
\end{equation}
In particular  the assumption   $\delta(\Gamma)\leq \frac{1}{2}d(X)$ is equivalent to $\lambda_0(M_\Gamma)= \lambda_0(X)$. Note that $\delta(\Gamma) = d(X)$ if and only if the discrete subgroup $\Gamma$ is co-amenable in the Lie group $G$. For example every lattice subgroup $\Gamma$ satisfies $\delta(\Gamma) = d(X)$.

The preceding discussion allows us to reformulate Theorem \ref{thm:not confined symmetric spaces} as follows.
\begin{cor}
\label{new_cor}
Let  $M$ be a locally symmetric space. Assume that the universal cover of $X$ is a rank one symmetric space. If $\lambda_0(M)= \lambda_0(X)$  then the space $M$ has unbounded injectivity radius.
\end{cor}
 
 Our result is sharp, at least for  the two-dimensional real hyperbolic plane $\mathbb{H}^2$.  Indeed  consider any  uniform lattice $\Gamma$ in the group of isometries $\Isom{\mathbb{H}^2} \cong \mathrm{PSL}_2(\mathbb{R})$.    It turns out that $\Gamma$ admits a sequence of infinite index normal subgroups $N_i$ satisfying $\delta(N_i) \to \frac{1}{2} \delta(\Gamma) = \frac{1}{2}$ \cite[Theorem 3.8]{Bonfert-Taylor-Matsuzaki-Taylor}. However every non-trivial normal subgroup of the lattice $\Gamma$ is confined  \cite[Example 1.2]{fraczyk2023infinite}.

Interestingly the condition $\delta(\Gamma) \le \frac{1}{2}d(X)$ for a discrete subgroup $\Gamma$ of the simple rank one Lie group $G$ is  equivalent to saying that the unitary  representation $L^2(G/\Gamma)$ is tempered \cite[Theorem 1.2]{edwards2023temperedness}. Furthermore, if the discrete subgroup $\Gamma$ is not confined then the unitary representation $L^2(G)$ is weakly contained in the unitary representation $L^2(G\backslash\Gamma)$ \cite[Proposition 8.4]{edwards2023temperedness}. We obtain the following.
\begin{cor}
Let $\Gamma$ be a discrete subgroup of the simple rank one Lie group $G$. If the unitary representation $L^2(G\backslash \Gamma)$ is tempered then the two unitary representations $L^2(G)$ and $L^2(G\backslash \Gamma)$ are weakly equivalent.
\end{cor}

%Consider the Chabauty space $\Sub{\Isom{X}}$ of all closed subgroups of   $\Isom{X}$. The group $\Isom{X}$ admits a continuous action on the space $\Sub{\Isom{X}}$ by conjugation.

\subsection*{Stationary  random subgroups and critical exponents}
While the statement of Theorem \ref{thm:not confined symmetric spaces} is purely geometric, the   methods we use are   probabilistic. Along the way we prove a number of probabilistic results of independent interest. These are considered in the more general setting of Gromov hyperbolic spaces.

Let $X$ be a proper  Gromov hyperbolic geodesic metric space with non-elementary group of isometries $\Isom{X}$. For example $X$ could be a rank one symmetric space or a  Gromov hyperbolic group equipped with the word metric. 
Let $\Sub{\Isom{X}}$ denote the Chabauty space of all closed subgroups of the group $\Isom{X}$. The group $\Isom{X}$ admits a continuous action  on this  space by conjugation.

Fix a Borel probability measure  $\mu$ on the group $\Isom{X}$ whose support generates as a semigroup a discrete subgroup of $\Isom{X}$ acting cocompactly on the space $X$. Associated to the   measure $\mu$ are the classical parameters 
 \emph{drift}   $l(\mu)$ and   \emph{entropy} $h(\mu)$  discussed in some detail in \S\ref{sec:random walks and poincare series}. The assumption that $\Isom{X}$ is non-elementary implies that $l(\mu) > 0$ as well as $h(\mu) > 0$. It is convenient to denote $\delta(\mu) = \frac{h(\mu)}{l(\mu)}$.

A \emph{$\mu$-stationary random subgroup} of the group $\Isom{X}$ is a Borel probability measure $\nu$ on the    space  $\Sub{\Isom{X}}$ satisfying $\mu * \nu = \nu$ with respect to the action by conjugation.  We establish a lower bound on the  critical exponents of discrete $\mu$-stationary random subgroups in terms of the parameter $\delta(\mu)$. 
% , namely
% \begin{equation}
% \nu = \int_G   \nu^g \; \mathrm{d}\mu(g)
% \end{equation}
% where $\nu^g(A) = \nu(A^{g^{-1}})$ for every element $g \in G$ and Borel subset $A \subset \Sub{\Isom{X}}$.

%A $\mu$-stationary random subgroup $\nu$ is called \emph{discrete} if $\nu$-almost every subgroup is discrete.

%\subsection*{Critical exponents}

%(see \S\ref{sec:hyperbolic groups}). 
%For instance if $X$ is a rank one symmetric space then $\mathrm{E}(\Isom{X})$ is the center of simple Lie group $\Isom{X}$.

\begin{theorem}
\label{thm:main thm intro}
Assume that the  probability measure $\mu$  has finite first moment.
Let $\nu$  be a  $\mu$-stationary random subgroup of $\Isom{X}$ such that $\nu$-almost every subgroup $\Gamma$ is discrete and not contained in the elliptic radical\footnote{The \emph{elliptic radical}  of the group $\Isom{X}$ is the maximal normal subgroup $\mathrm{E}(\Isom{X})$ consisting of elliptic elements. If $X$ is rank one symmetric space or more generally a geodesically complete $\mathrm{CAT}(-1)$-space then the elliptic radical of $\Isom{X}$ is trivial.}. Then
\begin{enumerate}
    \item $\delta(\Gamma) > \frac{\delta(\mu)}{2}$ holds $\nu$-almost surely, and
    \item if $\nu$-almost every subgroup $\Gamma$  is of divergence type\footnote{Critical exponents as well as the notion of divergence type subgroups are discussed in \S\ref{sec:critical exponents}.} then $\delta(\Gamma) \ge \delta(\mu)$ holds $\nu$-almost surely.
\end{enumerate}
\end{theorem}

%We note that when $X$ is a manifold of negative curvature bounded away from zero, the assumption that that $\nu$-almost every subgroup $\Gamma$ is discrete  not contained in the elliptic radical turns out to be vacuous. 

If $X$ is a rank one symmetric space then it is possible to choose a probability measure $\mu$ as above with finite first moment and with $\delta(\mu) = d(X)$. This fact is discussed in \S\ref{sec:critical exponents}. If $X$ is the Cayley graph of a rank $k$ free group then the uniform probability measure  $\mu_k$ on the standard symmetric generating set satisfies $\delta(\mu_k) = d(X) = \log(2k-1)$. This  provides a new proof of \cite[Theorem 1.1]{gekhtman2019critical} for invariant random subgroups in those two special cases\footnote{More generally, this remark applies to any   $\mathrm{CAT}(-1)$-space admitting a uniform lattice, see \S\ref{sec:critical exponents}.}.

More generally, if the group $\Isom{X}$ 
is acting properly and cocompactly on the space $X$ and $L$ is a uniform lattice in the group $\Isom{X}$ then there is a sequence of  probability measures $\mu_i$ supported on finite generating sets of the lattice $L$ satisfying  $\delta(\mu_i) \to \delta(L) = \dim_\mathrm{Haus}(\partial X)$.

The study of invariant random subgroups of discrete groups and the spectral radius of random walks on the associated Schreier graphs was initiated by Abert, Glasner and Virag   \cite{abert2014kesten}. Their work serves as the inspiration to ours.

\subsection*{On non-confined subgroups}

Consider a rank one simple Lie group $G$ with  a fixed Borel probability measure $\mu$.  Given any  discrete subgroup $\Gamma \le G$ let $\mathrm{D}_\Gamma$ denote the Dirac   mass supported on the point $\Gamma \in \Sub{G}$. Then  any accumulation point $\nu$ of the Ces\'{a}ro averages $\frac{1}{n}\sum_{i=1}^n \mu^{*i} * \mathrm{D}_{\Gamma}$ is a $\mu$-stationary random subgroup of $G$. At this point, the lower semi-continuity of the critical exponent with respect to the Chabauty topology would in principle allow us to deduce    Theorem \ref{thm:not confined symmetric spaces} from Theorem \ref{thm:main thm intro} along the lines of the strategy introduced in \cite{fraczyk2023infinite}. 
However, there is an additional missing ingredient, namely we need to know that $\nu$-almost every subgroup is discrete.
%Motivated by this idea, we  say that a probability measure $\mu$ on the   group $G$ is \emph{non-condensing} if any weak-$*$ accumulation point   of Ces\'{a}ro averages as above is almost surely discrete. 
%A key ingredient towards the proof of Theorem \ref{thm:not confined symmetric spaces} is Theorem \ref{thm:condition for a probability measure to be non-condensing} constructing a rich family of non-condensing probability measures on the group $G$ (including the discretization of Brownian motion).   
The actual proof is  more involved. In particular we   use  the \emph{key inequality} for the injectivity radius function from  \cite{GLM}. 

%By general considerations any such accumulation point is $\mu$ stationary, so the lower semi-continuity of $\delta(H)$ allows us to prove Theorem \ref{thm:not confined symmetric spaces}.

%We also establish the noncondensing property certain finitely supported measures on $Isom(X)$, namely uniform averages over sufficiently large balls in an orbit of a hyperbolic group.

The above strategy showing that a discrete subgroup of the isometry group of some Gromov hyperbolic space with a small critical exponent is not confined works more generally. See   Theorem \ref{thm:not confined general Gromov hyperbolic case}  for a precise formulation of a general principle. For example, in the case of free groups   we recover  \cite[Corollary 3.2]{fraczyk2020kesten}:
 
\begin{theorem}[Fraczyk]
\label{thm:not confined free groups}
Let $F_k$ be a free group of rank $k \in \mathbb{N}$ and  $X_k$ be its Cayley graph with respect to a  standard basis. If $H$  is any subgroup of $F_k$ with $\lambda_0(H\backslash X_k) = \lambda_0(X_k)$ then $H$ is not confined.
\end{theorem}

 Here $\lambda_0$ denotes the bottom of the spectrum of the combinatorial Laplacian defined on $2k$-regular graphs \cite[\S4.2]{lubotzky1994discrete}.
 In concrete terms, the conclusion of Theorem \ref{thm:not confined free groups} says that for any $R > 0$ there is some element $g \in F_k$ such that the conjugate $H^g$ contains no non-trivial elements of word length less  than $R$.

\subsection*{Limit sets}

Let $X$ be a proper  Gromov hyperbolic geodesic metric space. Lattices in the group of isometries $\Isom{X}$ as well as normal subgroups of such lattices have full limit sets. The same is true for invariant random subgroups of the group $\Isom{X}$ as was shown in \cite[Proposition 11.3]{abert2020growth} and \cite{osin2017invariant}. It turns out that  this fact extends to \emph{discrete} stationary random subgroups.

%almost every subgroup  Assume that the group $\Isom{X}$ is acting cocompactly on the space $X$. Any uniform lattice in $\Isom{X}$ as well as any normal subgroup 
%An action of a group on a Gromov hyperbolic space is of \emph{general type} if it has uncountable limit set and no fixed points on the boundary.  

\begin{theorem}
\label{thm:intro limit set}
Let $\mu$ be a Borel probability measure on the group $\Isom{X}$. Assume that the closed semigroup $G_\mu$ generated by $\mathrm{supp}(\mu)$ is a group whose action on the space $X$ has general type. 

Let $\nu$  be a discrete $\mu$-stationary random subgroup of $\Isom{X}$. If $\nu$-almost every subgroup is not contained in the elliptic radical $\mathrm{E}(\Isom{X})$ then the action of $\nu$-almost every subgroup $\Gamma$ has general type and its limit set $\Lambda(\Gamma)$ contains the limit set $\Lambda(G_\mu)$.\end{theorem}

In particular, if the group  generated by the support of the measure $\mu$ has full limit set $\Lambda(G_\mu) = \partial X$  then so does $\nu$-almost every subgroup. 

The discreteness assumption in Theorem \ref{thm:intro limit set}  is necessary in general, as  can be   seen  by considering the stabilizer of a random point with respect to the $G$-action on the  boundary at infinity $\partial X$ equipped with a $\mu$-stationary probability measure. 

Rich and detailed structural information on amenable non-elementary Gromov hyperbolic locally compact groups was obtained in the work \cite{caprace2015amenable}. These are precisely the groups admitting a proper and cocompact \emph{focal} action on a proper Gromov hyperbolic space, see \S\ref{sec:hyperbolic groups}. Taking this work into account and relying on our methods we study stationary random subgroups of such groups.

\begin{thm}
\label{thm:intro IRS of amenable hyperbolic}
Let $G$ be an amenable non-elementary Gromov  hyperbolic locally compact group and $\mu$  a Borel  probability measure on the group $G$. Assume that the measure $\mu$ has finite first moment and is spread out\footnote{A probability measure $\mu$ on a locally compact group $G$ is called \emph{spread out} is some convolution power $\mu^{*n}$ is non singular with respect to the Haar measure on the group $G$.}. Then every discrete $\mu$-stationary random subgroup of $G$ is contained in the elliptic radical $\mathrm{E}(G)$.
\end{thm}

 Theorem \ref{thm:intro IRS of amenable hyperbolic} applies for example in the case where the group $G$ is the stabilizer   of a boundary point at infinity in a rank one symmetric space or a regular tree. The general case   is much broader, as is manifested in \cite{caprace2015amenable}.

An immediate consequence of Theorem \ref{thm:intro IRS of amenable hyperbolic} is that amenable non-elementary Gromov hyperbolic  groups admit no discrete invariant random subgroups. 
This is a natural generalization of the fact such groups admit no lattices as they are never unimodular.
%This consequence in fact follows already from \cite{biringer2017unimodularity}.

\subsection*{$\CAT{0}$-spaces}

Let $X$ be a proper geodesically complete $\mathrm{CAT}(0)$-space such that the group of isometries $\Isom{X}$ is acting cocompactly on $X$. A subgroup $H$ of the group $\Isom{X}$ is called  \emph{geometrically dense} if $H$ preserves no proper closed convex  subset  of the space $X$ and fixes no point of the visual boundary  $\partial X$. 

Let $\mu$ be a  Borel probability measure on the group $\Isom{X}$ such that the closed semigroup generated by $\mathrm{supp}(\mu)$ is a geometrically dense subgroup.
Some of our methods are applicable to study $\mu$-stationary random subgroups of $\Isom{X}$.
% $\CAT{0}$-spaces.  

\begin{prop}

 Let $\nu$ be a $\mu$-stationary random subgroup of $\Isom{X}$. If $\nu$-almost every subgroup fixes no point of the boundary $\partial X$ then $\nu$-almost every subgroup is geometrically dense.
\end{prop}

This can be seen as a partial analogue of the work \cite{duchesne2015geometric} for stationary (rather than invariant) random subgroups,  conditional on the extra assumption of having no fixed points on the boundary. We expect that in future work this assumption can be replaced by a discreteness assumption (as was successfully done in the Gromov hyperbolic case in Theorem \ref{thm:intro limit set}). 
%combined with  irreducibility. 

At present, we are able to combine our  methods for Gromov hyperbolic spaces and $\mathrm{CAT}(0)$-spaces to  study stationary random subgroups of direct products of finitely many $\CAT{-1}$-spaces.

\begin{theorem}
\label{thm:intro product}
Let $X_1,\ldots,X_k$ be a family of proper geodesically complete $\CAT{-1}$-spaces.  Let $\mu_i$ be a  Borel probability measure on the group $\Isom{X_i}$ such that  $ \mathrm{supp}(\mu_i)$ generates $\Isom{X_i}$ as a semigroup for each $i$. Consider the product $\CAT{0}$-space $X = X_1 \times \cdots \times X_k$ and the  measure   $\mu = \mu_1 \otimes \cdots \otimes \mu_k$ on the group $\Isom{X}$. 

Let $\nu$ be a discrete $\mu$-stationary  random subgroup of $\prod_{i=1}^k \Isom{X_i}$. If   $\nu$-almost every subgroup projects non-trivially to each factor $\Isom{X_i}$ then $\nu$-almost every subgroup is geometrically dense.
\end{theorem}

Much more precise information in the special case of the action of a semisimple real Lie group  on its symmetric space is given in \cite[Theorem 6.5]{fraczyk2023infinite}.

\subsection*{Lower bounds on critical exponents}
As an illustration,   we outline a proof of the  non-strict variant of  Theorem \ref{thm:main thm intro}  asserting that a $\mu$-stationary random subgroup satisfies the non-strict inequality $\delta(\Gamma)\geq \frac{\delta(\mu)}{2}$ almost surely.
An analogue of Theorem \ref{thm:main thm intro} for \emph{invariant} rather than \emph{stationary} random subgroups with the bound $\frac{\delta(\mu)}{2}$ replaced by $d(X) = \frac{\mathrm{dim}_\mathrm{Haus}(\partial X)}{2}$ was the main result of \cite{gekhtman2019critical}.  
In turn, this  is  a generalization of a result of Matsuzaki, Yabuki and Jaerisch \cite[Theorem 1.4]{matsuzaki2020normalizer} about the exponential growth rate of orbits for a normal subgroup $N$  of some countable hyperbolic group $\Gamma$ acting properly  and cocompactly on a Gromov hyperbolic space $X$.

The key idea of \cite{matsuzaki2020normalizer} is to bound  the critical exponent $\delta(N)$ from below in terms of  the exponential growth rate of a single conjugacy class \begin{equation}
    \mathrm{Cl}_\Gamma(\gamma) = \{\gamma^g \: : \: g \in \Gamma \}
\end{equation}
of some fixed   arbitrary  hyperbolic element $\gamma \in N$. Fix a basepoint $x_0 \in X$ and denote $\|g\| = d_X(gx_0,x_0)$ for every element $g \in \Gamma$. The triangle inequality implies $\|g\gamma g^{-1}\|\leq 2\|g\|+\|\gamma\|$  for every element $g\in \Gamma$. So
\begin{equation}
\label{eq:first MYJ}
\{g\in \Gamma \: : \: \|g\|\leq \frac{R-\|\gamma\|}{2} \} \subset  \{g\in \Gamma \: :\: \|g\gamma g^{-1}\|\leq R\}  \quad \forall R > 0.
\end{equation}
On the other hand, the hyperbolic element $\gamma$ satisfies the following two properties:  its centralizer is  virtually cyclic   and $\lim \frac{1}{n} \|\gamma^n \| > 0$. This implies that
%n $\Gamma$ contains a finite index cyclic subgroup, namely the set of powers of $\gamma$, at most $C_\gamma R$ of which have displacement at most $R$. Consequently, we have 
\begin{equation}
\label{eq:second MYJ}
| \{  h\in \mathrm{Cl}_\Gamma(\gamma)  \: : \:  \|h\|\leq R \}|
\geq
\frac{c}{R} |\{g\in \Gamma \: : \: \|g\|\leq \frac{R-\|\gamma\|}{2} \}\| \quad \forall R > 0
\end{equation}
for some constant $c > 0$.
As linear factors do not affect the exponential growth rate we conclude from Equation (\ref{eq:second MYJ}) that $
\delta(N) \ge \delta(\mathrm{Cl}_\Gamma(\gamma)) \geq \frac{\delta(\Gamma)}{2}$.

%Indeed, the same argument shows that if $\gamma\in Isom(X)$ is any hyperbolic isometry and $\Gamma<Isom(X)$ a discrete cocompact subgroup, then the set $Conj_\Gamma(\gamma)$ of conjugates of $\gamma$ by elements of $\Gamma$ satisfies $\delta(Conj_\Gamma (\gamma)) \geq \delta(\Gamma)/2$. 

Consider a more general situation where $\nu$ is a discrete invariant random subgroup of the group of isometries $\Isom{X}$ of some Gromov hyperbolic space $X$.
Unlike normal subgroups of lattices,  there is a priori no reason for $\nu$-generic subgroups to contain an entire conjugacy class of any particular hyperbolic element.  Instead, we show in  \cite{gekhtman2019critical} that $\nu$-almost every subgroup contains a \emph{positive proportion} of some hyperbolic \enquote{approximate  conjugacy class}.
%However, we show in \cite{gekhtman2019critical} that  a \emph{uniform neighborhood} of an infinite discrete ergodic IRS $H<G$ contains a \emph{positive proportion} of a conjugacy class.
More precisely,   assume that the group $\Isom{X}$ admits some auxiliary discrete subgroup  $\Gamma$   acting cocompactly on the space $X$. Then for any $\varepsilon>0$ there exists an open relatively compact subset $V \subset \Isom{X} $ consisting of  hyperbolic elements such that $\nu$-almost every subgroup $\Delta \in \Sub{\Isom{X}}$ satisfies
\begin{equation}
\label{eq:ergodic theorem app}
\liminf_{R\to \infty}\frac{|\{g\in \Gamma \: :\: \|g\|\leq R \quad \text{and} \quad  \Delta \cap gVg^{-1}\neq \emptyset\}|}{|\{g\in \Gamma \: :\: \|g\|\leq R\}|} \ge 1-\varepsilon.
\end{equation} 
This statistical information is sufficient to conclude that $\delta(\Delta) \ge \frac{\delta(\Gamma)}{2}$ along the same lines as in \cite{matsuzaki2020normalizer}.

To find such a neighborhood $V \subset \Isom{X}$ with $\nu(\{\Delta  :  \Delta \cap V \neq \emptyset\})  > 1 - \varepsilon$ we first show that $\nu$-almost every subgroup admits hyperbolic elements, as  a consequence of the fact that it has a full limit set on the Gromov boundary. The maximal ergodic theorem of Bowen and Nevo \cite{bowen2015neumann}  applied to the probability measure preserving action of the auxiliary lattice $\Gamma$ on the Borel probability space $(\Sub{\Isom{X}},\nu)$  implies Equation (\ref{eq:ergodic theorem app}) directly.

\subsection*{On our methods}
Several important aspects of this strategy need to be modified  when $\nu$ is a \mbox{$\mu$-stationary} rather than an  invariant discrete random subgroup. First, the fact that $\nu$-almost every subgroup has full limit set and  hence contains hyperbolic elements  now follows from our Theorem \ref{thm:intro limit set}. Second, the maximal ergodic theorem of Bowen and Nevo does not apply to  stationary actions. Instead, we rely on \emph{Kakutani's random ergodic theorem}. Roughly speaking, it says that typical trajectories of the $\mu$-random walk become $\nu$-equidistributed. This implies  that $\mu^{\otimes \mathbb{N}}$-almost every sequence of random elements $(g_i) \in \Gamma^{\mathbb{N}}$ and $\nu$-almost every subgroup $\Delta \in \Sub{\Isom{X}}$ satisfy
\begin{equation}
\lim_{n\to\infty} \frac{| \{ i \in \{1,\ldots,n\} \: : \: \Delta \cap V^{g_i\cdots g_1} \neq \emptyset  \}|}{n} = \nu(\{ \Lambda \: : \Lambda \cap V \neq \emptyset \}).
\end{equation}
%for all 
%we have letting $\omega_n=g_n...g_1$ that
%$\frac{1}{N}\|i=1,...N: \omega_i\in \Omega_V\|\to \nu(\Omega_V)$.

To get a lower bound on the critical exponent $\delta(\Delta)$ for $\nu$-almost every subgroup $\Delta$  we need to estimate the ratios appearing in Equation (\ref{eq:ergodic theorem app}) for large $R$. In other words, we need to show that the subset $\{g\in \Gamma   :   H \cap V^g \neq\emptyset\}$ is  large from the point of view of taking intersections with balls rather than from the point of view of random walks. The connection between these two notions of \enquote{size} of subsets of the group $\Gamma$ is rather complicated (see for instance  \cite[Theorem 6.2]{tanaka2017hausdorff} for an example of  a \enquote{paradoxical} behaviour in this respect).

% showed that in any non-virtually free hyperbolic group $\Gamma$ with a left invariant metric $d$, for any $\epsilon>0$ there are subsets $A\subset \Gamma$ such that $\lim \inf \mu^\mathbb{N}((g_i):g_n...g_1\in A)>1-\epsilon$, but $|A\cap B_R|/|B_R|\to 0$ exponentially in $R$ where $B_R$ denotes the radius $R$ ball around the identity in $\Gamma$. In other words a set can be large from the point of view of random walks but quite small from the point of view of geometric averages. This is intimately tied to singularity of $\mu$-stationary and Patterson-Sullivan measures on the boundaries of hyperbolic groups, see e.g. \cite{tanaka2017hausdorff} \cite{gouezel2018entropy}.

In the setting of finitely supported random walks and word metrics on hyperbolic groups, Tanaka  showed in \cite[\S 6]{tanaka2017hausdorff} that any subset $A \subset \Gamma$  which contains the entire trajectory 
$\{\omega_n = g_1 \cdots g_n\}_{n\in \mathbb{N}}$
for an 
$\mu^{\otimes \mathbb{N}}$-positive measure  set of sequences $(g_i) \in \Gamma^\mathbb{N}$ has critical exponent $\delta(A) \ge \delta(\mu) = \frac{h(\mu)}{l(\mu)}$. We prove and make use of Theorem \ref{thm:improvedtanaka} which is a vast generalization of Tanaka's result. For us the group $\Gamma$ need not be hyperbolic, but merely non-amenable. The random walk need not be finitely supported, but only to have  finite first moment. This improvement  allows us to work  with measures having $\delta(\mu) = d(X)$, such as the discretization of Brownian motion. The set $A$ in question is assumed   to contain a \emph{positive proportion} of each random walk trajectory rather than the whole trajectory, i.e. sample paths visit it statistically rather deterministically. Finally, we conclude that  the partial Poincare series associated to the set $A$  diverges at the exponent   $\delta(\mu)$, which is stronger than concluding that $\delta(A)\geq \delta(\mu)$.

%While \cite{gekhtman2019critical}  relied on a delicate maximal ergodic theorem for hyperbolic groups  \cite{bowen2015neumann}, this paper uses much softer random walk methods. Our main  probabilistic argument works for any non-amenable group.

We expect   our results concerning critical exponents can be extended beyond isometry groups of hyperbolic spaces, for instance to rank-one $\CAT{0}$-spaces, Cayley graphs of relatively hyperbolic groups, and products of rank one symmetric spaces. The probabilistic framework for these generalizations is  laid out in this paper; one needs to modify a  few geometric techniques from \cite{gekhtman2019critical}.

\subsection*{Organization of the paper}

This work  consists of two largely independent parts. 

The first part \S\ref{sec:stationary measures}--\S\ref{sec:CAT0} deals with the softer questions of limit sets and fixed points on the boundary. In \S\ref{sec:stationary measures} we introduce the basics of stationary measures and random walks and state Kakutani's ergodic theorem which plays a fundamental  role in this work. In \S\ref{sec:hyperbolic groups} we recall the classification of group actions on Gromov hyperbolic spaces, including the focal and the general type cases. In \S\ref{sec:limit sets} we use boundary hitting measures to study stationary closed subsets. This information is applied to study limit sets. In \S\ref{sec:fixed points on the boundary} we show that discrete stationary subgroups cannot have a single fixed point on the boundary, and apply this to study amenable hyperbolic groups. Lastly \S\ref{sec:products} is dedicated to stationary random subgroups of products and \S\ref{sec:CAT0} deals with $\CAT{0}$-spaces and geometric density.

The second part \S\ref{sec:random walks and poincare series}--\S\ref{sec:random walks and confined subgroups}  deals with critical exponents and confined subgroups. The core   probabilistic argument takes place in  \S\ref{sec:random walks and poincare series}.  This is where we discuss entropy, drift, Green's functions and prove our variant of Tanaka's argument. Next in \S\ref{sec:critical exponents} we use this probabilistic machinery to study the Poincare series and the critical exponents of stationary random subgroups. 
%A standalone discussion of the key inequality and non-condensing probability measures for rank one simple Lie groups takes place in \S\ref{sec:non condensing}. 
Finally in \S\ref{sec:random walks and confined subgroups} we consider confined subgroups and implement our strategy towards Theorem \ref{thm:not confined symmetric spaces} by observing that the critical exponent is Chabauty semi-continuous.

\begin{remarkkk}
Since our work first appeared, an alternative geometric approach to critical exponents of confined subgroups has been developed by \cite{CGYZ} and \cite{DY}. In particular, these results give a version of Theorem  \ref{thm:not confined symmetric spaces} assuming either a \emph{strict} inequality $\delta(\Gamma)<\delta(X)/2$ or $\Gamma$ being contained in a lattice. However, despite holding in the broader geometric setting of actions with contracting elements, these results do not recover the full strength of Theorem \ref{thm:not confined symmetric spaces} and its Corollary \ref{new_cor}.
\end{remarkkk}

%\subsection*{Open problems}

%\begin{enumerate}
%\item  \end{enumerate}

\subsection*{Acknowledgements} 
The authors   would like to thank Uri Bader, Mikolaj Fraczyk, Sebastien Gouezel, Nir Lazarovich, Hee Oh, Yehuda Shalom, Ryokichi Tanaka and Giulio Tiozzo for  useful discussions, insightful comments and suggestions. We are very grateful to Inhyeok Choi, Peter Kosenko, Wenyuan Yang, and Tianyi Zheng for pointing out several inaccuracies in an earlier draft of this paper.
Special thanks are due to Tsachik Gelander who conjectured Theorem \ref{thm:not confined symmetric spaces} and brought it to our attention.

%about stationary random subgroups as well as  the problems discussed in this work,  for useful discussions about Poisson boundaries of locally compact groups,  for many insightful comments and  for explaining to us the proof of Lemma \ref{hittingrate}.

%\newpage
\subsection*{Notations} 

\begin{itemize}
\item
We denote $\left[n\right] = \{1,\ldots,n\}$ for every natural number $n \in \mathbb{N}$.
\item
We will often fix a base point $x_0 \in X$. With the base point being implicit in the notation, it will be convenient to write
\begin{equation*}
\|g\| = d_X(gx_0,x_0) \quad \forall g \in \Isom{X}.
\end{equation*}
\item To avoid an overuse of the Greek letter $\delta$ we will use $\mathrm{D}_x$ to denote the Dirac probability measure supported on the point $x$.
\item For a pair of functions $f$ and $g$ the notation $f \approx g$ means that $$C^{-1} f \le g \le C f$$ for some constant $C > 1$.
\item All probability measures are assumed to be Borel.
\end{itemize}

%\setcounter{tocdepth}{1}
%\tableofcontents

\newpage

\section{Stationary measures and Kakutani's ergodic theorem}
\label{sec:stationary measures}

We  introduce some general definitions concerning stationary measures and random walks, which are the main workhorse of this note. In addition we state Kakutani's ergodic theorem and summarize some of its implications.

\subsection*{Sample paths}
Let $G$ be a locally compact group. The product space $G^\mathbb{N}$ equipped with the Tychonoff topology will be regarded as describing  \emph{increments}.
Given  a sequence of increments $(g_n) \in G^\mathbb{N}$ the associated \emph{sample path} is the sequence $\omega= (\omega_n) \in G^\mathbb{N}$ defined by \begin{equation}
\omega_n = g_1 g_2 \cdots   g_n   \quad \forall n \in \mathbb{N}.
\end{equation}

Let $\mu$ be a  probability measure on the group $G$. The corresponding product measure on the   space of increments $G^\mathbb{N}$ is denoted  $\mu^{\otimes \mathbb{N}}$. Let $\mathrm{P}_\mu$ be the pushforward of the probability measure $\mu^{\otimes \mathbb{N}}$ via the mapping $(g_n) \mapsto (\omega_n)$. We will call $(G^{\mathbb{N}},\mathrm{P}_\mu)$ the \emph{space of sample paths}.

%The \emph{space of sample paths} is the product space $G^\mathbb{N}$ equipped with the Tychonoff topology. A \emph{sample path} is an element $\omega \in G^\mathbb{N}$. Given a sample path $\omega = (g_1,g_2,\ldots) \in G^\mathbb{N}$ we  denote 
%\begin{equation} \%omega_n = g_1 g_2 \cdots   g_n   \quad \forall n \in \mathbb{N}. \end{equation}

%We equip the space of sample paths $G^\mathbb{N}$ with the product probability measure $\mu^{\otimes \mathbb{N}}$. In particular, it makes sense to talk of $\mu^{\otimes \mathbb{N}}$-almost every sample path $\omega \in G^\mathbb{N}$. The value $\omega_n \in G$ is the position of the $\mu$-random walk after taking $n$ steps.

On certain occasions it will be useful to consider the  probability measure $\check{\mu}$ on the group $G$ given by
 $\check{\mu}(A) = \mu(A^{-1}) $
for every Borel subset $A \subset G$.
We will use the notation $\check{w}$ for sample paths distributed with respect to the measure $\mathrm{P}_{\check{\mu}}$. Note that
\begin{equation}
\check{\omega}_n = g_1^{-1} g_{2}^{-1}  \cdots g_n^{-1} = (g_n g_{n-1} \cdots  g_1)^{-1 } \quad \forall n \in \mathbb{N}.
\end{equation}
 The probability measure $\mu$ is \emph{symmetric} if $\mu=\check{\mu}$.
%So the distribution of $\check{\omega}_n$ is obtained by applying $\check{\cdot}$ to the distribution of $\omega_n$. 

%Consider the product probability measure $\mu^{\otimes \mathbb{N}}$ on the space of paths $G^\mathbb{N}$.

\subsection*{Stationary measures}
Let $(Z,\nu)$ be a standard  Borel probability space admitting a Borel action of the group $G$. Assume that the measure $\nu$ is \emph{$\mu$-stationary}, i.e. $\nu = \mu * \nu $ where
\begin{equation}
  \mu * \nu  = \int_G g_* \nu \; \mathrm{d}\mu(g).
\end{equation}
Assume further that $\nu$ is \emph{ergodic}, i.e. $\nu$ is an extreme point in the simplex of $\mu$-stationary probability measures.

The following ergodic theorem plays a central part in most of our arguments.

\begin{theorem}[Kakutani's ergodic theorem \cite{kakutani1951random}]
\label{thm:Kakutani ergodic theorem}
Let $f\in L^1(Z,\nu)$. Then $\mu^{\otimes \mathbb{N}}$-almost every sequence of increments $ (g_n) \in G^\mathbb{N}$  and $\nu$-almost every point $z \in Z$ satisfy 
\begin{equation}
\frac{1}{N}\sum^{N}_{n=1}f(g_n g_{n-1}\cdots g_1 z)\xrightarrow{N\to\infty}   \int f \, \mathrm{d} \nu.
\end{equation}
\end{theorem}

% We can reformulate Kakutani's ergodic theorem as follows.

% \begin{cor}
% \label{cor:consequence of Kakutani for inverses}
% Let $f\in L^1(Z,\nu)$.  Then  $\check{\mu}^\mathbb{N}$ almost every sequence  of elements $g_n$ and $\nu$-almost every point $z\in Z$ satisfy
% \begin{equation}
% \frac{1}{N}\sum^{N-1}_{i=0}f(\eta_i^{-1} z)\to \int f \, \mathrm{d} \nu
% \end{equation}
% where $\eta_n=g_1 g_2 \cdots g_n$.
% \end{cor}

% The reason this seemingly pedantic reformulation is useful is that as we will discuss later, in $\CAT{0}$ and Gromov hyperbolic settings typical sequences $g_1 \ldots g_n$ converge to points on the boundary whereas the sequences $g_n...g_1$ do not.

We will have  occasion to apply Kakutani's ergodic theorem in two ways, which we record here for convenience. First, it can be used to estimate the proportion of time a generic random walk spends inside some fixed subset of the space $Y$.

\begin{cor}
\label{cor:liminf of visits to positive measure set}
Let $Y \subset Z$ be a $\nu$-measurable subset. Then   
$\mu^{\otimes \mathbb{N}}$-almost every sequence of increments $ (g_n) \in G^\mathbb{N}$ and 
$\nu$-almost every point $z \in Z$ satisfy  
\begin{equation}
\lim_{n \to \infty} \frac{| \{ i\in \left[n\right]  \: : \: g_i g_{n-1} \cdots g_1 z \in Y\}| }{n} = \nu(Y).
\end{equation}
\end{cor}
\begin{proof}
The conclusion follows by applying  Kakutani's ergodic theorem (Theorem \ref{thm:Kakutani ergodic theorem}) to the characteristic function   $1_Y \in L^1(Z,\nu)$  of the subset $Y$.
% implies  that 
% \begin{equation}
% \frac{1}{N}\sum _{i=0}^{N-1} 1_Y(\omega_n z) \to \nu( Y)  > 0
% \end{equation}
% holds for almost every sample path $\omega$ and $\nu$-almost every point $z\in Z$. The desired conclusion follows.
\end{proof}

Second, Kakutani's theorem implies that a generic random walk \enquote{does not escape to infinity}, in the following sense.

\begin{cor}
\label{cor:liminf of a Borel function is bounded}
Let $F : Z \to \left[0,\infty\right)$ be a $\nu$-measurable  function. Then 
$\mu^{\otimes \mathbb{N}}$-almost every sequence of increments $ (g_n) \in G^\mathbb{N}$ and 
$\nu$-almost every point $z \in Z$ satisfy  
\begin{equation}
\liminf_{n\to\infty} F(g_n g_{n-1}\cdots g_1  z) < \infty.
\end{equation}
\end{cor}
\begin{proof}
Take $D > 0$ to be sufficiently large so that
\begin{equation}
\nu(Z_D) > 0 \quad \text{where} \quad Z_D = \{ z \in Z \: : \: F(z) \le D\}.     
\end{equation} 
The desired conclusion follows at once by applying Corollary \ref{cor:liminf of visits to positive measure set} with respect to the $\nu$-measurable subset $D$.
\end{proof}

\subsection*{Stationary measures are quasi-invariant}
 Let $H$ denote the closure in $G$ of the semigroup generated by the support of $\mu$.  It will be useful to keep in mind the following   well-known elementary fact.  We include a brief proof for completeness.
%  It is known for experts, see \cite{bader2006factor}, \cite[Chapter 12]{hasselblatt2002handbook} or \cite[arXiv version, Theorem 9.4]{gekhtman2021martin}. We briefly recall the proof.

\begin{lemma}
\label{lem:Stationary measures are quasi-invariant}
Assume that the topological group $G$ is second countable. Then any $\mu$-stationary measure $\nu$ is $H$-quasi-invariant, i.e. if a Borel subset $E \subset Z$ has $\nu(E)=0$ then $\nu(h^{-1}E)=0$ for any element $h\in H$.
\end{lemma}
\begin{proof}
Consider any $\mu$-stationary measure $\nu$ on the space $Z$. By induction $\mu^{*n} * \nu = \nu$ for all $n \in \mathbb{N}$. Let $E \subset Z$ be any $\nu$-null subset. Then
\begin{equation}
0 = \nu(E) = \int_G \nu(g^{-1}E) \, \mathrm{d}\mu^{*n}(g) \quad \forall n \in \mathbb{N}.
\end{equation}
As   $\nu$ is probability measure  it must satisfy $\nu(h^{-1}E) = 0$ for $\mu^{*n}$-almost every element $h \in G$ where $n \in \mathbb{N}$ is arbitrary.  The second countability assumption implies that for every $n \in \mathbb{N}$ there is a   dense subset $S_n \subset \mathrm{supp}(\mu^{*n})$ such that $\nu(h^{-1}E) = 0$ for all elements $h \in S_n$. Denote $H_0 = \bigcup_n S_n$ so that $\overline{H}_0 = H$.  The conclusion follows from that fact that the subset of $H$ consisting of all elements $h$ with $\nu(h^{-1} E) = 0$ is closed by \cite[Theorem B.3]{zimmer2013ergodic} applied to the Koopman representation.
\end{proof}

In particular, if $Z$ is a topological space and $\nu$ is a  $\mu$-stationary Borel measure  then $\mathrm{supp}(\nu)$ is an $H$-invariant closed subset of $Z$.

% \subsection*{Non-existence criterion for stationary measures}

% Let $G$ be a locally compact topological group. In certain situations actions of the group $G$ on non-compact topological spaces admit no stationary Borel probability measures. 

% \begin{lemma}
% \label{lemma:on almost stationary}
% Let $\varphi : G \to \mathbb{R}$ be a continuous homomorphism and $\mu$ be a Borel probability measure on the group $G$ such that $\varphi_* \mu \neq \mathrm{D}_0$ and $\varphi_* \mu$ has finite second moment. Let $X$ be a topological space admitting a continuous $G$-action. Assume that there are constants $a > 1$ and $b  > 0$ and a continuous function $f : X \to \mathbb{R}$ satisfying
% \begin{equation}
% f(x) + a^{-1} \varphi(g) - b \le f(gx) \le f(x) + a\varphi(g) + b
% \end{equation}
% for all elements $g\in G$ and points $x\in X$. Then the space $X$ admits no   $\mu$-stationary Borel probability measures.
% \end{lemma}

% Note that the above assumptions  force the space $X$ to be non-compact.

% \begin{proof}[Proof of Lemma \ref{lemma:on almost stationary}]
% Assume towards contradiction that $\nu$ is a $\mu$-stationary Borel probability measure on the topological space $X$. Take $R > 0$ to be sufficiently large so that $\nu(\{x \in X \: : \: |f(x)| \le R\}) > 0.9$. The classical central limit theorem gives some  $n \in \mathbb{N}$ so that $\mu^{*n}(\{g \in G \: : \: |\varphi(g)| > a(R + b) \} ) > 0.9 $. Put together  these two estimates contradict the fact that the measure $\nu$ is $\mu$-stationary so that $\mu^{*n} * \nu = \nu$.
% \end{proof}

\section{Group actions on hyperbolic spaces}
\label{sec:hyperbolic groups}

%\marginpar{Do the results in this section depend on properness?}
Let $X$ be a  Gromov hyperbolic geodesic metric space\footnote{In this section we do not require the metric space $X$ to be proper (unless explicitly stated otherwise).}. We recall various features of groups acting on $X$. In addition we discuss several relevant notions, namely quasi-stabilizers and the elliptic radical.

\subsection*{Classification of actions} 
The \emph{Gromov boundary} $\partial X$ is the collection of geodesic rays in the space $X$ up to the equivalence relation of being within bounded Hausdorff distance. The group $\Isom{X}$ is acting on the boundary $\partial X $ by homeomorphisms. The space $X \cup \partial X$ is a natural  compactification of the space $X$.

Let $G$ be a group admitting a continuous action on the space $X$ by isometries. Equivalently we are given a continuous homomorphism $f: G \to \Isom{X}$. The \emph{limit set} $\Lambda(G)$ is the set $\partial X \cap \overline{Gx_0}$ for any choice of base point $x_0 \in X$.

Recall our standing notation $\|g\| = d_X(gx_0,x_0)$ where $x_0 \in X$ is any fixed base point. The action of any single element  $g \in G$ on the space $X$ can be classified as
\begin{itemize}
\item \emph{elliptic} if $g$ has bounded orbits.
\item \emph{parabolic} if $g$ has unbounded orbits and $\lim_{n\to\infty} \frac{1}{n} \|g^n\| = 0$.
\item \emph{hyperbolic} if  $\lim_{n\to\infty} \frac{1}{n} \|g^n\| > 0$.
\end{itemize}
In terms of the limit set of the cyclic group generated by an element $g \in G$ we can say that
\begin{itemize}
\item $g$ is elliptic if and only if $\Lambda(\left<g\right>) = \emptyset$.
\item $g$ is parabolic if and only if $|\Lambda(\left<g\right>)| = 1$.
\item $g$ is hyperbolic if and only if $|\Lambda(\left<g\right>)| = 2$.
\end{itemize}
In addition parabolic and hyperbolic elements satisfy $\Lambda(\left<g\right>) = \mathrm{Fix}_{\partial X}(g)$. 

Any action of a group on the space $X$ can be classified as follows \cite[\S8]{gromov1987hyperbolic}.
\begin{itemize}
\item The action is called \emph{elementary} if it is
\begin{itemize}
\item   \emph{bounded:}   orbits are bounded,
\item   \emph{horocyclic:}  it is not  bounded but  has no hyperbolic isometries, or
\item  \emph{lineal:}  it has an hyperbolic isometry and any two hyperbolic isometries have the same limit set.
\end{itemize}
\item The action is called \emph{non-elementary} if it is
\begin{itemize}
\item \emph{focal:} it has a hyperbolic isometry, is not lineal and the limit sets of any two hyperbolic isometries are not disjoint.
\item of \emph{general type:}  it has two hyperbolic isometries with disjoint limit sets.
\end{itemize}
\end{itemize}
We refer to \cite[Proposition 3.1]{caprace2015amenable} for a list of equivalent conditions for each of the above types of actions in terms of the limit set. For instance, a locally compact group is focal hyperbolic with respect to the word metric coming from some compact generating set if and only if it is amenable and non-elementary hyperbolic \cite[Theorem 7.3]{caprace2015amenable}.

\subsection*{The elliptic radical}
Let $G$ be a group admitting a continuous non-elementary action on the space $X$ by isometries. The \emph{elliptic radical} $\mathrm{E}(G)$ of the group $G$ is the kernel of its action on its limit set, namely it is the closed normal subgroup
\begin{equation}
\mathrm{E}(G) = \{ g \in G \: : \: g\zeta = \zeta \quad \forall \zeta \in \Lambda(G)\}.
\end{equation} 
Alternatively, the elliptic radical $\mathrm{E}(G)$ can be characterized as the maximal normal subgroup of $G$ whose action on the space $X$ is bounded \cite[Lemma 3.6]{caprace2015amenable}. See also \cite[Proposition 3.4]{osin2017invariant}.

\subsection*{Quasi-stabilizers}

Let $\Gamma$ be a discrete group admitting an action by isometries on the Gromov hyperbolic space $X$.

\begin{defn}
\label{def:quasi stabilizer}
Given a  point $x \in X$ and a radius $R > 0$ the corresponding \emph{quasi-stabilizer} is given by
\begin{equation}
\Gamma_{x,R} = \{\gamma \in \Gamma \: : \: d_X(\gamma x, x) \le R \}.
\end{equation}
\end{defn}
Note that the actual stabilizer $\Gamma_x = \{\gamma \in \Gamma \: : \: \gamma x = x\}$ is contained in the quasi-stabilizer $\Gamma_{x,R}$ for all $R > 0$. The quasi-stabilizer is not  a subgroup in general. 
In the situation where the metric space $X$ is proper, we know that the action of the group $\Gamma$ on the metric space $X$ is proper if and only if all quasi-stabilizers are finite. 

\begin{lemma}
\label{lem:points with small inj radius}
There is some sufficiently large $R > 0$  such that if $\gamma \in \Gamma$ is  any elliptic or parabolic element then $\Gamma_{x,R} \neq \emptyset$ for some  point $x \in X$.
\end{lemma}
The  lemma  is certainly well-known, we include it for the sake of completeness. It is immediate  from the definitions if the space $X$ is $\mathrm{CAT}(-1)$ rather than $\delta$-hyperbolic.

\begin{proof}[Proof of Lemma \ref{lem:points with small inj radius}]
 If the element $\gamma$ under consideration is parabolic then this fact is discussed in \cite[8.1.C]{gromov1987hyperbolic}. If the element $\gamma$ is elliptic we argue as follows. Let $O$ be any orbit of the isometry $\gamma$ in the space $X$. The orbit $O$ is bounded. Denote
 \begin{equation}
 r_O = \inf \{ r > 0 \: : \: \text{$O \subset B_x(r)$ for some $x \in X$} \}
 \end{equation}
 and
 \begin{equation}
 C_1(O) = \{ x \in X \: : \: O \subset B_x(r_O + 1) \}.
 \end{equation}
 The points of   $C_1(O)$ are called the \emph{quasi-centers} of the subset $O$. The subset $C_1(O)$ is clearly $\gamma$-invariant and has diameter  less than $4\delta + 2$, see \cite[III.$\Gamma$.3.3]{bridson2013metric}. 
 The claim follows by taking  an arbitrary point $x \in C_1(O)$.
\end{proof}

\section{Stationary closed subsets}
\label{sec:limit sets}

Let $X$ be a complete separable Gromov hyperbolic geodesic metric space with fixed base point $x_0 \in X$ and equipped with the metric $d_X$. For example $X$ could be a proper Gromov hyperbolic length space (see \cite[\S I.3]{bridson2013metric}). Generally speaking, and with some applications in mind, we do not want to assume $X$ is proper.

\subsection*{Boundary measures for random walks}
Let $G$ be a locally compact group admitting a continuous action by isometries on the space $X$. Fix a   probability measure $\mu$ on the group $G$. Assume that the closed semigroup $G_\mu$ generated by the support of the measure $\mu$ is in fact a group. Consider the associated  space of sample paths $(G^\mathbb{N},\mathrm{P}_\mu)$  discussed in \S\ref{sec:stationary measures}.
The following fundamental object plays a key role in our work.

\begin{theorem}[The boundary measure]
\label{thm:boundary measure}
Assume that the action of the group $G_\mu$ on the space $X$ is non-elementary. In the focal case assume in addition  that
\begin{itemize}
    \item the measure $\mu$ has finite first moment\footnote{The probability measure $\mu$ has \emph{finite first moment} if $\sum_{g \in G} \mu(g) \|g\| < \infty$.} and is spread out\footnote{A measure $\mu$ is called \emph{spread out} if some convolution power $\mu^{*n}$ is non singular  with respect to the Haar measure on the group $G$.}, 
    \item the Poisson boundary of the pair $(G,\mu)$ is non trivial, and
    \item the space $X$ is proper and the group $G$ is acting cocompactly on $X$.
\end{itemize}
% \begin{itemize}
%     \item the action of the group $G_\mu$ on the space $X$ is of general type, or
%     \item the metric space $X$ is proper, the measure $\mu$ has finite first moment, some convolution power $\mu^{*n}$ is absolutely continuous with respect to Haar measure and the group $G_\mu$ is acting cocompactly on $X$.
% \end{itemize}
In the general type case no additional assumption is required. Then
\begin{enumerate}
\item the trajectory $\omega_n x_0$ converges to some point $\zeta_\omega \in \partial X$ for $\mathrm{P}_\mu$-almost every sample path $\omega  \in G^\mathbb{N}$, and
\item the pushforward $\nu_X$ of the measure $\mathrm{P}_\mu$ via the mapping $\omega \mapsto \zeta_\omega$ is the unique $\mu$-stationary  probability measure  on the boundary $\partial X$ satisfying  $\nu_X(\mathrm{Fix}_{\partial X}(G_\mu)) = 0$.
\end{enumerate}
\end{theorem}
Note that the limit of a sequence such as $\omega_n x_0$ is  independent of the base point. The condition saying that $\nu_X(\mathrm{Fix}_{\partial X}(G_\mu)) = 0$ is clearly  relevant only in the focal case, for a general type action fixes no points on the boundary at infinity.
\begin{proof}[Proof of Theorem \ref{thm:boundary measure}]
Assume to begin with that the action of the group $G_\mu$ on the metric space $X$ has general type.  The case of a random walk on a Gromov hyperbolic group is \cite{kaimanovich1994poisson}.   The general case for a proper metric space $X$ is  treated in \cite{benoist2016central}. The case where the group $G$ is countable but the space $X$ may be non-proper is \cite{maher2018random}. However, as noted by Dussaule and Gouezel, the countability assumption may be dropped. Indeed, the only place where it was used in \cite{maher2018random} was in their Proposition 4.4, which Dussaule and Gouezel generalized to the uncountable setting, see  \cite[Appendix]{Dussaule2016notes}.

In the focal case, the extended  set of assumptions implies     that the Poisson boundary of the pair $(G,\mu)$ can be identified with the Gromov boundary $\partial X$ equipped with the  $\mu$-stationary hitting measure $\nu_X$, see \cite[Theorem 1.4]{forghani2022shannon}. 
See \cite{kaimanovich1996boundaries} for a closely related  survey of the identification problem for Poisson boundaries. The argument of \cite[Proposition 3.1]{benoist2016central} shows that $\nu_X$ is the unique \emph{non-atomic} $\mu$-stationary probability measure  on the boundary $\partial X$. On the other hand, the only  atomic such measure is the Dirac measure $\mathrm{D}_\zeta$ where $\mathrm{Fix}_{\partial X}(G_\mu) = \{\zeta\} \subset  \Lambda(G_\mu)$. 
\end{proof}

\emph{We will assume throughout \S\ref{sec:limit sets} and \S\ref{sec:fixed points on the boundary} that the assumptions of Theorem \ref{thm:boundary measure} are satisfied so that the $\mu$-stationary boundary measure $\nu_X$ satisfying both statements (1) and (2) of Theorem \ref{thm:boundary measure} does exist.}

\vspace{5pt}

 For the sake of completeness we include the proof of the following elementary consequence of Theorem \ref{thm:boundary measure}.

\begin{lemma}
\label{lemma:identifying conv powers}
The  probability measures $\mu^{*n} * \mathrm{D}_{x_0}$ converge to the  $\mu$-stationary boundary measure $\nu_X$ in the weak-$*$ topology of probability measures on the compactification $\overline{X} = X \cup \partial X$.
\end{lemma}
\begin{proof}
It follows from Theorem \ref{thm:boundary measure} that 
\begin{equation}
\nu_X(E) = \mathrm{P}_\mu(\{\omega  \: : \: \lim_n \omega_n x_0 \in E\})
\end{equation}
holds true for any $\nu_X$-measurable subset $E \subset \partial X$. Take $U$ to be any open subset of the compactification $\overline{X}$. If $\lim_n \omega_n x_0 \in U \cap \partial X$ then $\omega_n x_0 \in U$ for all indices $n$ sufficiently large. Fatou's lemma gives 
\begin{align}
\begin{split}
\nu_X(U) &= \nu_X(U \cap \partial X) \le 
\int_{G^\mathbb{N}} \liminf_n  1_U (\omega_n x_0) \; \mathrm{d} \mathrm{P}_\mu(\omega) \le \\
&\le \liminf_n \int_X 1_U \, \mathrm{d}(\mu^{*n} * \mathrm{D}_{x_0}) = \liminf_n \mu^{*n} * \mathrm{D}_{x_0}(U).
\end{split}
\end{align}
The desired conclusion follows from the Portmanteau theorem. 
\end{proof}

\begin{prop}
\label{prop:stationary measure on boundary of Gromov hyperbolic space has full support}
The boundary measure $\nu_X$ has  $\mathrm{supp}(\nu_X) = \Lambda(G_\mu)$.
\end{prop}

\begin{proof} 
The construction of the boundary measure $\nu_X$ in Theorem \ref{thm:boundary measure} shows that  $\mathrm{supp}(\nu_X) \subset \Lambda(G_\mu)$. The subset $\mathrm{supp}(\nu_X)$ is closed and $G_\mu$-invariant by Lemma \ref{lem:Stationary measures are quasi-invariant}. This concludes the proof since $\Lambda(G_\mu)$ is the minimal $G_\mu$-invariant closed subset of the boundary $\partial X$ not contained in $\mathrm{Fix}_{\partial X}(G_\mu)$ \cite[Corollary 7.4.3]{das2017geometry}.
% Assume towards contradiction that $\mathrm{supp}(\nu_X)$ is a proper subset of $\Lambda(G_\mu)$. Therefore there is a non-empty open subset   $U \subset \Lambda(G_\mu)$ with  $\nu_X(U) = 0$. As the action of the   subgroup $G_\mu $ is of general type there is some loxodromic element  $g \in G_\mu$ with $g^+ \in U$.  The quasi-invariance of the measure $\nu_X$ implies that $\nu_X(\bigcup_{n \in \mathbb{Z}} g^n U) = 0$. We conclude that the measure $\nu_X$ is supported on the singleton  $\Lambda(G_\mu)  \setminus \bigcup_{n \in \mathbb{Z}} g^n U = \{g^{-1}\}$. This is a contradiction.
\end{proof}

%A closed non-empty subset $Y \subset X$ is called \emph{proper at infinity} if $\partial Y \subsetneq \partial Y$.

%\subsection*{Stationary random subsets}

\subsection*{Effros space of closed subsets}
Let $\Cl{X}$ denote the space of all closed subsets of the space $X$. The space $\Cl{X}$ is a standard Borel space with respect to the \emph{Effros Borel structure} given by the $\sigma$-algebra generated by all subsets of the form
\begin{equation}
\mathcal{O}(U) = \{ F \in \Cl{X} \: : \: F \cap U = \emptyset\}
\end{equation}
for some open subset $U \subset X$ \cite[Theorem 12.6]{kechris2012classical}.  In the special case where the metric space $X$ is proper the Effros Borel structure coincides with the Borel structure for the Chabauty topology on $\Cl{X}$. For our purposes in \S\ref{sec:limit sets} it will suffice to regard $\Cl{X}$ as a  Borel rather than a topological space.

The action of the group $\Isom{X}$ on the space $X$ by isometries is continuous. It determines a Borel action of the group $G$ on the Effros  space  $\Cl{X}$.

\begin{prop}
\label{prop:distance from a proper at infinity subset goes to infinity with positive probability}
Let $Y \in \Cl{X}$ be a closed subset satisfying  $\Lambda(G_\mu) \setminus \partial Y \neq \emptyset$. Then there is a subset of sample paths $\Omega \subset G^\mathbb{N}$   with $\mathrm{P}_\mu(\Omega) > 0$ so that
\begin{equation}
\label{eq:sample path leaves proper at infinity subsets}
\liminf_n d_X(\omega_n x_0,Y) = \infty  
\end{equation}
holds true for every sample path $\omega \in \Omega$.
\end{prop}
\begin{proof}
The   open subset $U = \Lambda(G_\mu) \setminus \partial Y$ satisfies $\nu_X(U) > 0$  by  Proposition \ref{prop:stationary measure on boundary of Gromov hyperbolic space has full support}.     Let $\Omega = \{ \omega \in G^\mathbb{N} \: : \: \zeta_\omega \in U \}$ so that $\mathrm{P}_\mu(\Omega) > 0$. Note that $\mathrm{P}_\mu$-almost  every sample path $\omega \in \Omega$ satisfies $\omega_n x_0 \to \zeta_\omega$ for some point $\zeta_\omega \notin \partial Y$. Every such  sample path $\omega$ must in paticular satisfy Equation (\ref{eq:sample path leaves proper at infinity subsets}).
\end{proof}

\subsection*{Stationary closed subsets} 

Denote $\Lambda_\mu = \Lambda(G_\mu)$.  We use the above information to deduce that $\mu$-stationary random closed subsets of the space $X$ have large limit sets.

\begin{prop}
\label{prop:no stationary measure on proper at infinity subsets}
Let $\nu$ be a $\mu$-stationary probability measure on $\mathrm{Cl}(X)$. Then $\nu$-almost  every  non-empty closed subset $Y \in \Cl{X}$    satisfies $\Lambda_\mu \subset \partial Y$.
\end{prop}
\begin{proof}
Consider the Borel function 
\begin{equation}
    Y \mapsto d_X(x_0,Y), \quad \Cl{X}\setminus \{\emptyset\} \to \left[0,\infty\right)
\end{equation}
defined on the space of closed non-empty subsets of $X$. Corollary \ref{cor:liminf of a Borel function is bounded} says that
\begin{equation}
\label{eq:distance stays bounded}
\liminf_n d_X(x_0, g_n \cdots g_1 Y) < \infty
\end{equation}
 for $\mu^{\otimes \mathbb{N}}$-almost every sample sequence of increments $ (g_n) \in G^{\otimes \mathbb{N}}$ and $\nu$-almost every closed non-empty subset $Y \subset X$. The desired conclusion follows from Equation (\ref{eq:distance stays bounded}) combined with Proposition  \ref{prop:distance from a proper at infinity subset goes to infinity with positive probability} applied with respect to the probability measure $\check{\mu}$ given by $\check{\mu}(A) = \mu(A^{-1})$.
\end{proof}

\subsection*{Limit sets and fixed points on the boundary}

Given an arbitrary closed subset   $A \subset \partial X$ we let   $\mathrm{WH}(A)$ denote the union of all bi-infinite geodesics in $X$ with both endpoints belonging to $A$. In particular   $\mathrm{WH}(A) \in \Cl{X}$. Note that $\mathrm{WH}(A) \neq \emptyset$  if and only if $|A| \ge 2$. 

\begin{prop}
\label{prop:if limit set larger than 2 then it contains Lambda (G)}
Let $\nu$ be a $\mu$-stationary random subgroup of $G$. If $\nu$-almost every subgroup $H$ of $G$ satisfies $|\Lambda(H)| \ge 2$ then   $\Lambda_\mu  \subset \Lambda(H)  $ holds $\nu$-almost surely.
\end{prop}

The assumption that a subgroup $H \le G $ satisfies $|\Lambda(H)| \ge 2$ is equivalent to saying that $H$ contains an element acting on the space $X$ via a loxodromic isometry.

\begin{proof}[Proof of Proposition \ref{prop:if limit set larger than 2 then it contains Lambda (G)}]
Consider the $G$-equivariant map $F_1 : \Sub{G} \to \Cl{X}$ given by 
\begin{equation}
F_1(H) = \mathrm{WH}(\Lambda(H)).
\end{equation}
  The assumption   $|\Lambda(H)| \ge 2$ implies that $F_1(H) \neq \emptyset$.
Note that if
  $\Lambda(H) \cap \Lambda_\mu \subsetneq \Lambda_\mu$ then $ \partial \mathrm{WH} (H) \cap \Lambda_\mu \subsetneq  \Lambda_\mu$. The result follows from Proposition \ref{prop:no stationary measure on proper at infinity subsets}.
\end{proof}

\begin{prop}
\label{prop:fixed point set has size at most 1}
Let $\nu$ be a $\mu$-stationary random subgroup of $G$. If $\nu$-almost every subgroup $H$ is not contained in the elliptic radical $\mathrm{E}(G_\mu)$ then $|\mathrm{Fix}_{\Lambda_\mu}(H)| \le 1$ holds  true $\nu$-almost surely.
\end{prop}
\begin{proof}
Consider the $G$-equivariant map $F_2 : \Sub{G} \to \Cl{X}$ given by 
\begin{equation}
F_2(H) = \mathrm{WH}(\mathrm{Fix}_{\Lambda_\mu}(H)).
\end{equation}
If the subgroup $H$ is $\nu$-almost surely not contained in the elliptic radical $\mathrm{E}(G_\mu)$ then $\mathrm{Fix}_{\Lambda_\mu}(H)$ is $\nu$-almost surely a proper subset of the limit set $\Lambda_\mu$. In that case it follows from the contrapositive direction of Proposition \ref{prop:no stationary measure on proper at infinity subsets} that $F_2(H) = \emptyset$ must hold  $\nu$-almost surely. In other words $|\mathrm{Fix}_{\Lambda_\mu}(H)| \le 1$ holds  $\nu$-almost surely.
\end{proof}

\begin{prop}
\label{prop:limit set cannot be empty}
Let $\nu$ be a $\mu$-stationary random subgroup of $G$. If $\nu$-almost every subgroup $H$ is not contained in the elliptic radical $\mathrm{E}(G_\mu)$ then   $\Lambda(H) \neq \emptyset$ holds true $\nu$-almost surely.
\end{prop}
\begin{proof}
Assume towards contradiction that the condition $\Lambda(H) = \emptyset$ is not  $\nu$-null.  Every subgroup $H$ with $\Lambda(H) = \emptyset$ has bounded orbits \cite[Proposition 3.1]{caprace2015amenable}.  Fix some radius $R > 0$ and consider the function $F_3 : \Sub{G} \to \Cl{X}$ given by
\begin{equation}
F_3(H) = \{ x \in X \: : \: \sup_{h \in H} d_X(hx,h) \le R \}.
\end{equation}
Provided that the radius $R$ is taken to be sufficiently large we  have 
\begin{equation}
\nu(\{H \in \Sub{G} \: : \: F_3(H) \neq \emptyset \}) > 0    
\end{equation}
According to Proposition \ref{prop:no stationary measure on proper at infinity subsets} we deduce that $\nu$-almost every subgroup $H$ with $F_3(H) \neq \emptyset$ satisfies $\Lambda_\mu \subset \partial F_3(H)$. Every such subgroup must be  contained in the elliptic radical $\mathrm{E}(G_\mu)$. We arrive at a contradiction.
\end{proof}

Putting together the above information we obtain the following.

\begin{cor}
\label{cor:limit set is full or has size 1}
Let $\nu$ be a $\mu$-stationary random subgroup of $G$. If $\nu$-almost every subgroup $H$ is not contained in the elliptic radical $E(G_\mu)$ then     either $\Lambda_\mu \subset \Lambda(H) $ or $|\Lambda(H)| = 1$ holds true $\nu$-almost surely.
\end{cor}

\subsection*{Essential freeness of the boundary action}

We conclude this section with a nice application to be used in \S\ref{sec:products} below. Recall that $\Lambda_\mu = \Lambda(G_\mu)$ is the limit set and  $\mathrm{E}(G_\mu)$ is the elliptic radical of the group $G$.

\begin{cor}
\label{cor:about freeness of action on boundary}
Let $\nu$ be a $\mu$-stationary random subgroup of $G$. Assume that $|\mathrm{Fix}_{\Lambda_\mu}(H)| = 1$ holds true for $\nu$-almost every subgroup $H \in \Sub{G}$. Then for $\nu$-almost every subgroup $H \in \Sub{G}$ the quotient group $H / (H \cap   \mathrm{E}(G_\mu))$ acts essentially freely on the measure space $(\Lambda_\mu, \nu_X)$.
\end{cor}
\begin{proof}
Let $(\Sub{G} \times \Lambda_\mu,\eta)$ be the \emph{$\mu$-join}\footnote{The reason to use the notion of $\mu$-join is that the direct product of two $\mu$-stationary spaces may not be $\mu$-stationary with respect to the diagonal action. See \cite[\S 3]{furstenberg2010stationary} for details.} of the  two $\mu$-stationary spaces $(\Sub{G},\nu)$ and $(\Lambda_\mu, \nu_X)$. This means that  the probability measure $\eta$ is $\mu$-stationary and its marginals on the two factors $\Sub{G}$ and $\Lambda_\mu$ are $\nu$ and $\nu_X$, respectively.   There is a $G$-equivariant $\eta$-measurable map $\Psi : \Sub{G} \times \Lambda_\mu \to \Sub{G}$ given by $\Psi(H,\zeta) = H_\zeta$ where $H_\zeta = \mathrm{stab}_H(\zeta)$ for all subgroups $H \in \mathrm{Sub}(G)$ and boundary points $\zeta \in \partial X$. The pushforward measure $\Psi_* \eta$ is therefore a  $\mu$-stationary random subgroup of $G$. Note that $|\mathrm{Fix}_{\Lambda_\mu}(H)| \ge 2$ is true $\Psi_* \eta$-almost surely. So  $\Psi_* \eta$-almost every stabilizer subgroup $H_\zeta$ belongs to the elliptic radical $\mathrm{E}(G_\mu)$ according to Proposition \ref{prop:fixed point set has size at most 1}. The desired conclusion follows.
\end{proof}

\section{Geometric density for discrete stationary random subgroups}
\label{sec:fixed points on the boundary}

 We have seen in the previous  \S\ref{sec:limit sets} that the limit set of a $\mu$-stationary random subgroup is either the entire boundary $\partial X$ or a singleton. It is easy to see that both possibilities may occur in general. Our current goal is to exclude the singleton case for \emph{discrete} $\mu$-stationary random subgroups.

We maintain the notations and assumptions   of   \S\ref{sec:limit sets}. In particular $G$ is a locally compact group admitting a continuous action by isometries on a proper Gromov hyperbolic space $X$ and $\mu$ is a probability measure on the group $G$. The semigroup $G_\mu$ generated by $\mathrm{supp}(\mu)$ is a non-elementary group and there exists a $\mu$-stationary boundary measure as in Theorem \ref{thm:boundary measure}. In addition we assume that the group $G_\mu$ has full limit set, namely $\Lambda(G_\mu) = \partial X$.

\begin{thm}
\label{thm:a discrete SRS has full limit set}
Let $\nu$ be a $\mu$-stationary random subgroup of $G$. Assume that $\nu$-almost every subgroup   $\Gamma$ acts \emph{properly} on $X$ and is not contained in the elliptic radical $\mathrm{E}(G_\mu)$.
Then $\Lambda(\Gamma) = \partial X$ holds true $\nu$-almost surely.
\end{thm}

The proof of Theorem \ref{thm:a discrete SRS has full limit set} will be given  towards  the end of \S\ref{sec:fixed points on the boundary} below, after we develop some more machinery.

\subsection*{Quasi-stabilizers and fixed   points on the boundary}

The notion of quasi-stabilizers $\Gamma_{x,R}$ 
 was introduced in Definition \ref{def:quasi stabilizer}. 
For every radius $R > 0$ and every $B \in \mathbb{N}$ we denote
\begin{equation}
Q(\Gamma;R,B) = \overline{\{ x \in X \: : \: |\Gamma_{x,R}| > B \}} 
\end{equation}
so that $Q(\Gamma;R,B) \in \Cl{X}$. Roughly speaking, if the group $\Gamma$ fixes a point on the boundary at infinity then the sets $Q(\Gamma;R,B)$ are coarse analogues of horoballs.

\begin{lemma}
\label{lemma:Q is non-empty}
Let $\Gamma$ be a subgroup of $G$  acting on the space $X$ without  loxodromic isometries. There is some  constant $R > 0$ such that
\begin{equation}
\mathrm{Fix}_{\partial X}(\Gamma) \subset \partial Q(\Gamma;R,B) 
\end{equation}  
for  every $B \in \mathbb{N}$ with $B < |\Gamma|$. 
\end{lemma}

The constant $R$ in Lemma \ref{lemma:Q is non-empty} depends only on the space $X$ and is independent of the particular subgroup $\Gamma$.

\begin{proof}[Proof of Lemma \ref{lemma:Q is non-empty}]
%Assume without loss of generality that $\mathrm{Fix}_{\partial X}(\Gamma) \neq \emptyset$ for otherwise there is nothing to prove.
Consider an arbitrary boundary point   $\zeta \in \mathrm{Fix}_{\partial X}(\Gamma)$   fixed by the action of the group $\Gamma$.  Let $R_0$ be the constant provided by Lemma \ref{lem:points with small inj radius}.

Fix some $B \in \mathbb{N}$ with $B < |\Gamma|$ as well as  an arbitrary subset $\{\gamma_1,\ldots,\gamma_{B+1}\} \subset \Gamma$. 
 Since every element of the group $ \Gamma $ is acting on the space $X$ via a parabolic or an elliptic isometry there are  points $x_1,\ldots,x_{B+1} \in X$ satisfying $\gamma_i \in  \Gamma_{x_i,R_0}$ for all $i \in \{1,\ldots,B+1\}$.
 Take  a geodesic ray   $l_i$  from the point $x_i$ to the boundary point $\zeta$ for each $i \in \{1,\ldots,B+1\}$. There is some constant $R_1 = R_1(R_0,\delta)$ such that any point $y \in l_i$ satisfies $\gamma_i \in \Gamma_{y,R_1}$.
 %$d_X(\gamma_i y,y) < \delta'$.

 There is some point $z_B \in X$ within a distance of $5\delta$ from all of the   geodesic rays $l_1,\ldots,l_{B+1}$ \cite[III.H.3.3.(2)]{bridson2013metric}. The previous paragraph implies that $\gamma_i \in \Gamma_{z_B,R}$  for all indices $i \in \{1,\ldots,B+1\}$ provided that $R$ is a constant satisfying $R > 10\delta + R_1$.  In other words $z_B \in Q(\Gamma;R,B)$. In particular the subset in question $Q(\Gamma;R,B)$ is non-empty. 
 Lastly note that  the point $z_B$ can be taken in an arbitrary small neighborhood of the boundary point $\zeta$. It follows that  $\zeta \in \partial Q(\Gamma;R,B)$ as required.
\end{proof}

%Note that a discrete subgroup $\Gamma$ satisfies $|\Lambda(\Gamma)| = 1$ if and only if $\Gamma$ is infinite and $|\mathrm{Fix}_{\Gamma}(\partial X)| = 1$.

\begin{lemma}
\label{lemma:quasi stabilizers nested}
Let $\Gamma$ be a subgroup of $G$  acting on the space $X$ without  loxodromic isometries and fixing the point $\zeta_0 \in \partial X$. Let $x_n \in X$ be a sequence of points converging to some boundary point $\zeta \in \partial X$ distinct from $\zeta_0$. Then for each $ R > 0$ there is a point $z \in X$ and some $S>0$ such that
$
\Gamma_{x_n,R} \subset \Gamma_{z,S}
$
for all $n \in \mathbb{N}$.
\end{lemma}
\begin{proof}
Let $R > 0$ be arbitrary. For each $n \in \mathbb{N}$ take  a geodesic ray 
 $l_n$ from the point $x_n$ to the fixed boundary point $\zeta_0$.  The sequence of geodesic rays $l_n$ converges uniformly on compact sets to a bi-infinite geodesic line $l$ from the point $\zeta_0$ to the point $\zeta$. Let $z \in l$ be an arbitrary point on the bi-infinite geodesic $l$. There is a sufficiently large constant $S > 0$ depending on the sequence $x_n$, on the constant $R$ and on the hyperbolicity constant of the space $X$ such that $
\Gamma_{x_n,R} \subset \Gamma_{z,S}
$ for all $n \in \mathbb{N}$.
\end{proof}

\begin{lemma}
\label{lemma:boundary of Q is contained in large points}
Let $\Gamma$ be a subgroup of $G$  acting on the space $X$ without  loxodromic isometries and satisfying $\mathrm{Fix}_{\partial X}(\Gamma) \neq \emptyset$. Fix some constants $R > 0$ as well as $B \in \mathbb{N}$ with $B < |\Gamma|$. Then
\begin{equation}
\partial Q(\Gamma;R,B) \subset \{\zeta \in \partial X \: : \: |\Gamma_\zeta| > B \}.
\end{equation}
\end{lemma}
\begin{proof}
Consider a point $\zeta \in \partial Q(\Gamma;R,B)$. We wish to show that $|\Gamma_{\zeta}| > B$. Let $\zeta_0 \in \mathrm{Fix}_{\partial X}(\Gamma)$ be an arbitrary boundary point fixed by the action of $\Gamma$.
 We may  assume without loss of generality that $\zeta \neq \zeta_0$ for otherwise $|\Gamma_{\zeta}| = |\Gamma| > B$ in which case  we are done. 

Take a sequence of points   $x_n \in Q(\Gamma;R,B)$  converging to the boundary point $\zeta$. According to Lemma \ref{lemma:quasi stabilizers nested} there is some point $z \in X$ and some constant $S > 0$ such that $\Gamma_{x_n,R} \subset \Gamma_{z,S}$ for all $n \in \mathbb{N}$.  In other words, the subsets $\Gamma_{x_n,R}$ are all contained in the finite subset $\Gamma_{z,S}$ and satisfy $|\Gamma_{x_n,R}| > B$. Therefore 
\begin{equation}
|\limsup_n \Gamma_{x_n,S}| > B \quad \text{and} \quad \limsup_n \Gamma_{x_n,S} \subset \Gamma_\zeta.
\end{equation}
We conclude that  $|\Gamma_\zeta| > B$ as required.
\end{proof}

%The two Lemmas \ref{lemma:Q is non-empty} and \ref{lemma:boundary of Q is contained in large points} imply that with the suitable assumptions $|$$\mathrm{Fix}_{\partial X}(\Gamma)$ conclude from the two

Note that in the three previous lemmas (unlike the next lemma) the action of the subgroup $\Gamma$ on the space $X$ was \emph{not} assumed to be proper.

\begin{lemma}
\label{lemma:discrete group has finite stabilziers}
Let $\Gamma$ be a subgroup of $G$ acting on the space $X$ properly   and without loxodromic isometries. If   $\zeta_0 \in \mathrm{Fix}_{\partial X}(\Gamma) $ then $|\Gamma_\zeta| < \infty$ for every $\zeta \in \partial X \setminus \{\zeta_0\}$.
\end{lemma}
\begin{proof}
This statement follows immediately from \cite[Lemma 3.5]{caprace2015amenable}.
% Assume towards contradiction  that $|\Gamma_\zeta| = \infty$ for some boundary point $\zeta \in \partial X \setminus \{\zeta_0\}$.  In particular the group $\Gamma$ is infinite. Let $l$ be any  bi-infinite geodesic in the space $X$ connecting the two boundary points $\zeta_0$ and $\zeta$. Fix an arbitrary point $x \in l$. Any point in the  $\Gamma$-orbit of the point $x$   lies on some geodesic from $\zeta_0$ to $\zeta$ as well as on the same horosphere at $\zeta_0$ as the point $x$. Since any two  bi-infinite geodesic lines between a pair of distinct boundary points lie within a finite Hausdorff distance, we conclude that the $\Gamma$-orbit of the point $x$ is bounded. This is a contradiction to the assumption that the group $\Gamma$ is acting properly.
\end{proof}

\subsection*{Stationary random subgroups acting properly}

We have previously established in \S\ref{sec:limit sets} that a $\mu$-stationary random subgroup $\nu$ 
not contained in the elliptic radical $\mathrm{E}(G_\mu)$ satisfies $|\mathrm{Fix}_{\partial X}(H)| \le 1$ for $\nu$-almost every subgroup $H$. We  now refine this result under the additional assumption of discreteness.
%Let $\mathrm{DSub}(G)$ denote the Chabauty subspace consisting of discrete subgroups.

\begin{prop}
\label{prop:a discrete SRS cannot have a single fixed point on the boundary}
Let $\nu$ be a $\mu$-stationary random subgroup of $G$. If $\nu$-almost every subgroup $\Gamma$ is acting properly on the space $X$ and is not contained in the elliptic radical $\mathrm{E}(G_\mu)$ then  $|\mathrm{Fix}_{\partial X} (\Gamma)| \neq 1$ holds true $\nu$-almost surely.
\end{prop}
\begin{proof}
Throughout the proof it will be easier to assume without loss of generality that the measure $\nu$ is ergodic.

Assume towards contradiction  that $\nu$-almost every subgroup $\Gamma$ acts properly on $X$ and satisfies $\mathrm{Fix}_{\partial X}(\Gamma) = \{\zeta_\Gamma\}$ for some boundary point $\zeta_\Gamma \in \partial X$. The point $\zeta_\Gamma$ depends on the subgroup $\Gamma$ in a $\nu$-measurable manner.

We know from Proposition \ref{prop:limit set cannot be empty} that  $\Lambda(\Gamma) \neq \emptyset$ holds $\nu$-almost surely. In particular $\nu$-almost every subgroup $\Gamma$ is infinite. By Gromov's classification of actions on hyperbolic spaces,  the action of  $\nu$-almost every subgroup $\Gamma$ is either horocyclic or focal (see \S\ref{sec:hyperbolic groups}). In the horocyclic case the action of the group $\Gamma$ has no loxodromic isometries, and we may proceed with the proof as is. 

In the focal case we first argue as follows. Consider the map $\theta : \Sub{G} \to \Sub{G}$ given by $\theta(\Gamma) = \overline{\left[\Gamma,\Gamma\right]}$. The push-forward measure $\nu_1 = \theta_* \nu$ is a $\mu$-stationary subgroup of $G$. The action of  $\nu'$-almost every subgroup is  proper, horocyclic  and  fixes  at least one   point on the boundary (see \cite[Corollary 3.9]{caprace2015amenable}).  Proposition \ref{prop:fixed point set has size at most 1} says that  $\nu'$-almost every subgroup $\Gamma' = \theta(\Gamma)$ is either contained in the elliptic radical $\mathrm{E}(G_\mu)$ or satisfies $\mathrm{Fix}_{\partial X}(\Gamma')= \mathrm{Fix}_{\partial X}(\Gamma) = \{\zeta_\Gamma\}$.  The first possibility is eliminated by the fact that a focal action  admits a pair of loxodoromic isometries $\gamma_1, \gamma_2$ with $\gamma_1^+ = \gamma_2^+ = {\zeta_\Gamma}$  but $\gamma_1^- \neq \gamma_2^-$. 
Up to replacing the $\mu$-stationary random subgroup $\nu$ by $\nu_1 = \theta_* \nu$ in the focal case, we may assume for the remainder of the proof that $\nu$-almost every subgroup $\Gamma$ acts without loxodromic isometries.

For every subgroup $\Gamma \in \Sub{G}$  consider  the stabilizer sizes $|\Gamma_\zeta|$ as a function from the boundary $\partial X$ to the countable set $\mathbb{N} \cup \{\infty\}$. Clearly  $|\Gamma_{\zeta_\Gamma}| = \infty$ holds true $\nu$-almost surely. On the other hand $|\Gamma_\zeta| < \infty$ for every $\zeta \in \partial X $ with $\zeta \neq \zeta_\Gamma$ and for $\nu$-almost every subgroup $\Gamma$, see Lemma \ref{lemma:discrete group has finite stabilziers}.
 The Baire category theorem implies that there is some minimal $B \in \mathbb{N}$ such that the subset $\{\zeta \in \partial X \: : \: |\Gamma_\zeta| = B\} $ of the boundary $\partial X$ has non-empty interior. As the measure $\nu$ was assumed to be ergodic the same $B$ works for $\nu$-almost every subgroup $\Gamma$. In fact, it is $\nu \times \nu_X$-almost surely the case that the pair $(\Gamma,\zeta)$ satisfies $\overline{ \{\zeta \in \partial X \: : \: |\Gamma_\zeta| = B\}}  = \partial X$, for otherwise the convex hull of the  complementary subset of the boundary can be used to reach a contradiction to Proposition \ref{prop:no stationary measure on proper at infinity subsets}.

Consider the $\nu$-measurable mapping 
\begin{equation}
\Phi : \Gamma \mapsto Q(\Gamma;R,B), \quad \Phi : \Sub{G} \to \Cl{X}
\end{equation}
where $B \in \mathbb{N}$ is the constant determined in the previous paragraph and $R > 0$ is sufficiently large as in Lemma \ref{lemma:Q is non-empty}.  
   On the one hand $\zeta_\Gamma \in \partial Q(\Gamma;R,B)$ holds $\nu$-almost surely according to Lemma \ref{lemma:Q is non-empty}. On the other hand $\partial Q(\Gamma;R,B) \subsetneq \partial X$ occurs $\nu$-almost surely because of  the particular choice of $B$ and of Lemma \ref{lemma:boundary of Q is contained in large points}. Therefore $\Phi_* \nu$-almost every closed subset  is  non-empty and has proper boundary at infinity.   We arrive at a contradiction to Proposition \ref{prop:no stationary measure on proper at infinity subsets}.
\end{proof}

We are ready to conclude the following proof as a   consequence of Proposition \ref{prop:a discrete SRS cannot have a single fixed point on the boundary}.

\begin{proof}[Proof of Theorem \ref{thm:a discrete SRS has full limit set}]
Let $\nu$ be a $\mu$-stationary random subgroup of $G$ such   that $\nu$-almost every subgroup   $\Gamma$ acts properly on $X$ and is not contained in the elliptic radical $\mathrm{E}(G_\mu)$.
We know that  $\Lambda(\Gamma) = \partial X$ or $|\Lambda(\Gamma)| = 1$ holds $\nu$-almost surely, see Corollary \ref{cor:limit set is full or has size 1}. 
In the first case we are done. In the second case  $\Lambda(\Gamma) \subset \mathrm{Fix}_{\partial X}(\Gamma)$ so that $|\mathrm{Fix}_{\partial X}(\Gamma)| \ge 1$ holds $\nu$-almost surely. 
%However as $\Lambda(\Gamma) \neq \emptyset$ the subgroup $\Gamma $ is $\nu$-almost surely not contained in the elliptic radical $\mathrm{E}(G_\mu)$. 
On the other hand $|\mathrm{Fix}_{\partial X}(\Gamma)| \le 1$ according to Proposition \ref{prop:fixed point set has size at most 1}. Putting everything  together implies that  $|\mathrm{Fix}_{\partial X}(\Gamma)| = 1$ holds $\nu$-almost surely. This contradicts Proposition \ref{prop:a discrete SRS cannot have a single fixed point on the boundary}.
\end{proof}

\begin{cor}
\label{cor:SRS of general type are general type}
Assume that the action of the subgroup $G_\mu$ on the space $X$ is of general type. Let $\nu$ be a $\mu$-stationary random subgroup of $G$. If $\nu$-almost every subgroup acts properly on the space $X$ and is not contained in the elliptic radical $\mathrm{E}(G_\mu)$ then the action of $\nu$-almost every subgroup is of general type.
\end{cor}

\subsection*{Amenable hyperbolic groups}
\label{sec:amenable}

Caprace, de Cornulier, Monod and Tessera \cite{caprace2015amenable} developed a rich and detailed structure theory for   amenable  non-elementary Gromov hyperbolic locally compact groups. The action of a   hyperbolic group on itself is focal if and only if the group is  amenable and non-elementary hyperbolic. Every such group can be written as $\mathbb{R} \ltimes_\alpha H$ or $\mathbb{Z} \ltimes_\alpha H$ where $\alpha = \alpha(1)$ is a \emph{compacting} automorphism of the locally compact group $H$, namely there is some compact subset $V \subset H$ such that every element $g \in H$ satisfies $\alpha^n(g)\in V$ for all sufficiently large $n$. Moreover such a group  acts  properly and cocompactly by isometries on some proper geodesically complete $\CAT{-1}$-space.

\begin{theorem}
\label{thm:amenable hyperbolic new}
Let $G$ be an amenable non-elementary   hyperbolic locally compact group and $\mu$  a probability measure on the group $G$. Assume that the measure $\mu$ has finite first moment and is spread out. Then every discrete $\mu$-stationary random subgroup of $G$ is contained in the elliptic radical $\mathrm{E}(G_\mu)$.
\end{theorem}
\begin{proof}
% The action of the group $G$ on the space $X$ is focal.

First consider the case where  the Poisson boundary of the pair $(G,\mu)$ is not trivial. This means that the action of the group $G$ on itself by isometries gives rise to  a $\mu$-stationary boundary measure $\nu_G$ supported on the limit set  $\Lambda(G_\mu) \subset \partial G$ by the focal case of Theorem \ref{thm:boundary measure}. Unless $\nu$-almost every discrete subgroup $\Gamma$ is contained in the elliptic radical $\mathrm{E}(G)$, we get that $\mathrm{Fix}_{\partial G}(\Gamma) = \emptyset$ for $\nu$-almost every discrete subgroup $\Gamma$ according to the two Propositions \ref{prop:fixed point set has size at most 1} and \ref{prop:a discrete SRS cannot have a single fixed point on the boundary}. This is a contradiction to the fact that this action is focal.

Alternatively, consider the case where the Poisson boundary of the pair $(G,\mu)$ is trivial. This means that any $\mu$-stationary probability measure is in fact $G$-invariant, see e.g. \cite{furman2002random}. In particular the probability  measure $\nu$ is a discrete invariant random subgroup. 
We may find some compactly supported spread out probability measure $\mu'$ on the group $G$ such that the Poisson boundary of the pair $(G,\mu')$ is \emph{non} trivial. This is possible by \cite[Theorem 3.16]{jaworski1995strong} and by taking into account the fact that the locally compact group $G$ is non-elementary Gromov hyperbolic and as such has exponential growth. The invariant random subgroup $\nu$ can be regarded as a $\mu'$-stationary random subgroup. At this point we may conclude the proof exactly as in the previous paragraph.
\end{proof}

\begin{remark}
The stabilizer  of a  point at infinity for a general type action of a locally compact group  is a special case of an amenable non-elementary hyperbolic group. For such groups Theorem \ref{thm:amenable hyperbolic new} can be proved directly by inducing stationary random subgroups. The proof given above  is much more general however.
\end{remark}

\begin{remark}
If $G$ is any locally compact group with modular function $\Delta_G$ and $\nu$ is a discrete invariant random subgroup of $G$ then $\nu$-almost every subgroup is contained in $\ker \Delta_G$ \cite[Corollary 1.2]{biringer2017unimodularity}. In the context of amenable hyperbolic groups $\ker \Delta_G = H$. This  can be used to get a  weaker conclusion.
\end{remark}

%An action on a Gromov hyperbolic space is called \emph{geometrically dense} if it has full limit set and no fixed points on the boundary. The above result shows that a discrete stationary random subgroup is acting geometrically dense.
 
 % \begin{remark}
% Some of the above arguments are slightly less technical if the space $X$ is assumed to be a $\CAT{\kappa}$-space for some $\kappa < 0$ rather than a hyperbolic space.
% \end{remark}

\section{Discrete stationary random subgroups of products}
\label{sec:products}

Let $X$ be a proper Gromov hyperbolic geodesic metric space for which the action of $\Isom{X}$ is of general type. Let $Y$ be any proper metric space. Fix a pair of arbitrary base points $x_0 \in X$ and $y_0 \in Y$.   Regard the product  $X \times Y$ as a proper metric space with the $L^2$-product metric. Consider the group $G$ of isometries of the space $X \times Y$ preserving each factor, namely
\begin{equation}
G = \Isom{X} \times \Isom{Y}.
\end{equation}
Let $p_X$ and $p_Y$ denote the projection homomorphisms from the group $G$ to its coordinates $\Isom{X}$ and $\Isom{Y}$ respectively. The group $G$ is acting by isometries on each factor $X$ and $Y$ via the projections $p_X$ and $p_Y$.

%Finally denote $\mu = \mu_X \otimes \mu_Y$ so that $\mu$ is a probability measure on the group $G$.

 Let $\DSub{G}$ denote the Chabauty space of all discrete subgroups of the group $G$. Given a discrete subgroup $\Gamma \in \DSub{G}$ it will be convenient to introduce the shorthand notation
 \begin{equation}
  F_X(\Gamma) = \mathrm{Fix}_{\partial X}(\Gamma)
 \end{equation}
 where the group $\Gamma$ is understood to act on the space $X$ via the projection $p_X$.
 
 Fix a probability measure $\mu_X$ on the  group $\Isom{X}$.  Assume that $\mathrm{supp}(\mu_X)$  generates as a semigroup a dense subgroup of $\Isom{X}$.
 
 % In other words $\nu$ is a discrete $\mu$-stationary random subgroup of $G$.
  
 %f $F_i(\Gamma) = \partial X_i$ then $p_i(\Gamma)$ belongs to the elliptic radical $E(\mathrm{Isom}(X_i))$. We will assume that $\nu$-almost surely $F_i(\Gamma)$ is a proper subset of $\partial X_i$.

\begin{theorem}
\label{thm:trivial fixed points in products}
Let $\nu$ be a $\mu_X$-stationary probability measure on the space $\DSub{G}$.
Then 
either $F_X(\Gamma) = \emptyset$ or $p_X(\Gamma)$ is contained in the elliptic radical $\mathrm{E}(\Isom{X})$   for $\nu$-almost every subgroup $\Gamma$.
\end{theorem}
\begin{proof}
We assume for the sake of convenience and   without loss of generality that the measure $\nu$ is ergodic. 

Assume that the projection $p_X(\Gamma)$ is  $\nu$-almost surely \emph{not} contained in the elliptic radical $\mathrm{E}(\Isom{X})$. In this case we know by Proposition \ref{prop:fixed point set has size at most 1} that $|F_X(\Gamma)| \le 1$ holds true $\nu$-almost surely. Our goal is to show  that in fact $F_X(\Gamma) = \emptyset$ holds $\nu$-almost surely.

There are two cases to consider, depending on whether the action of $\nu$-almost every subgroup $\Gamma$ on the factor $X$ is proper or not. In the proper case $F_X(\Gamma) = \emptyset$ holds true  $\nu$-almost surely  by Proposition \ref{prop:a discrete SRS cannot have a single fixed point on the boundary} and we are done. For the remainder of the proof assume  that we are in the non-proper case, namely  the projection  $p_X(\Gamma)$ is $\nu$-almost surely a non-discrete subgroup of $\Isom{X}$, and that  $F_X(\Gamma) = \{\zeta_\Gamma\}$ for some point $\zeta_\Gamma \in \partial X$ (depending measurably on the subgroup $\Gamma$). We will  arrive at a contradiction.

To begin with, and up to replacing the random subgroup $\nu$ by its pushforward with respect to the commutator closure map $\Gamma \mapsto \overline{\left[\Gamma,\Gamma\right]}$, we may assume that $\nu$-almost every subgroup admits no loxodromic elements in its action on the factor $X$ via the projection $p_X$. See the proof of Proposition \ref{prop:a discrete SRS cannot have a single fixed point on the boundary} for more details concerning this argument.

Fix an arbitrary   radius $R > 0$. For each discrete subgroup $\Gamma \in \DSub{G}$ consider the quasi-stabilizer
\begin{equation}
\Gamma_{x_0,R} = \{\gamma \in \Gamma \: : \: d_{X}(p_X(\gamma) x_0, x_0) \le R \}
\end{equation}
corresponding to the action of the subgroup $\Gamma$ on the factor $X$ via the projection map $p_X$. This quasi-stabilizer is $\nu$-almost surely infinite as the projection $p_X(\Gamma)$ is non-discrete. 

Let $\omega  \in \Isom{X}^\mathbb{N}$ be a $\mathrm{P}_{\mu_X}$-random sample path. We claim that 
% \begin{equation}
% \limsup_n \Gamma^{g_n g_{n-1}\cdots g_1 } _{x_i, R} = \{e\}.
% \end{equation}
%  Indeed, the statement in question is equivalent to saying that 
\begin{equation}
\limsup_n \, p_X (\Gamma_{\omega_n x_0, R}) \subset \mathrm{E}(\Isom{X}).
\end{equation}
Indeed the sequence of points $\omega_n x_0$ converges $\mathrm{P}_{\mu_X}$-almost surely to the boundary point $\zeta_\omega \in \partial X$ satisfying moreover $\zeta_\omega \neq \zeta_\Gamma$.  We rely on the fact that any element  $\gamma \in \Gamma$ such that $p_X(\gamma)$ is not contained in the elliptic radical $\mathrm{E}(\Isom{X})$ is acting essentially freely on $(\partial X, \nu_{X})$, see Corollary \ref{cor:about freeness of action on boundary}. Therefore 
\begin{equation}
\liminf_n d_{X} \left(p_X(\gamma) \omega_n x_0, \omega_n x_0 \right) = \infty 
\end{equation}
$\mathrm{P}_{\mu_X}$-almost surely and for any element $\gamma \in \Gamma$ with $p_X(\gamma) \notin \mathrm{E}(\Isom{X})$. In particular any such element $\gamma$ must satisfy 
\begin{equation}
\gamma \notin \limsup_n \Gamma_{\omega_n x_0, R}.
\end{equation}
The claim follows. 

The discreteness of the subgroup $\Gamma$ implies that $p_Y(\Gamma_{x,R})$ is $\nu$-almost surely a discrete subset of $\Isom{Y}$ for any point $x \in X$. 
We define the Borel function
\begin{equation}
f : \DSub{G} \to \left[0,\infty\right), \quad f(\Gamma) = \inf_{\substack{\gamma \in \Gamma_{x_0,R} \\ \gamma \notin \mathrm{E}(G)}}  d_{Y}(p_Y(\gamma) y_0,y_0).
\end{equation}
Given an element $g \in \Isom{X} \le G$ we obviously have that $p_Y(g)y_0 = y_0$ and so
\begin{equation}
\label{eq:conjugating stuff}
f(\Gamma^g) = \inf_{\substack{\gamma \in \Gamma^g_{x_0,R} \\ \gamma \notin \mathrm{E}(G)}}  d_{Y}(p_Y(\gamma) y_0,y_0) = \inf_{\substack{\gamma \in \Gamma_{g^{-1} x_0,R} \\ \gamma \notin \mathrm{E}(G)}}  d_{Y}(p_Y(\gamma) y_0, y_0).
\end{equation}
We   remark that for $\mathrm{P}_\mu$-almost every sample path $\omega$ there is some point $z \in X$ and some constant $S > 0$ such that the subsets $p_Y(\Gamma_{\omega_n x_0,R})$ are all contained in the discrete subset $p_Y(\Gamma_{z,S}) \subset \Isom{Y}$, see Lemma \ref{lemma:quasi stabilizers nested}.  This remark, combined with   the preceding claim  applied for each fixed value $R > 0$ and with Equation (\ref{eq:conjugating stuff}), show that 
\begin{equation}
\liminf_n f( \Gamma^{g_n g_{n-1}\cdots g_1}) = \infty 
\end{equation}
for $\nu$-almost every subgroup $\Gamma$ and $\mu_X^{\otimes \mathbb{N}}$-almost every sequence of increments $(g_n) \in \Isom{X}^\mathbb{N}$. We emphasize that the sample path positions $\omega_n$ all belong to the factor $\Isom{X}$.  We arrive at a contradiction in light of  Kakutani's ergodic theorem,  see Corollary \ref{cor:liminf of a Borel function is bounded} specifically. 

We conclude that   $F_X(\Gamma) = \emptyset$ holds $\nu$-almost surely, as required.
\end{proof}

\subsection*{Products of hyperbolic spaces}

Fix some $k \in \mathbb{N}$. Let $X_1,\ldots,X_k$ be a collection of proper Gromov hyperbolic geodesic metric spaces. Assume that the action of the group $\Isom{X_i}$ on the space $X_i$ has  general type for each $i$. Let $X$ denote the product space 
\begin{equation}
X = X_1 \times \cdots \times X_n
\end{equation}
endowed with the product metric. 

Let    $G$ denote  the finite index subgroup  of the full group of isometries $ \Isom{X}$ consisting of these isometries that preserve each individual  factor $X_i$ for $i\in \{1,\ldots,k\}$. In other words
\begin{equation}
G = \Isom{X_1} \times \cdots \times \Isom{X_k}.
\end{equation}
Let $p_i$ denote the   projection homomorphism from the group $G$ to the coordinate $\Isom{X_i}$ for each $ i\in \{1,\ldots,k\}$. The group $G$ is acting on each factor $X_i$ via the projection $p_i$.

Fix    a probability measure $\mu_i$ on the group $\Isom{X_i}$ for each $i \in \{1,\ldots,k\}$. Assume that $\mathrm{supp}(\mu_i)$ generates as a semigroup a dense subgroup of $\Isom{X_i}$. Finally denote 
\begin{equation}
\mu = \mu_1 \otimes \cdots \otimes \mu_k.
\end{equation}

\begin{cor}
 Let $\nu$ be a discrete $\mu$-stationary random subgroup of $G$. If the projection $p_i(\Gamma)$ is $\nu$-almost surely not contained in the elliptic radical $\mathrm{E}(\Isom{X_i})$ for each $i$ then the action of $\nu$-almost every subgroup $\Gamma$ on the factor $X_i$ has general type.
\end{cor}
\begin{proof}
Note that the measure $\nu$ is $\mu_i$-stationary for each $i$ by \cite[Lemma 3.1]{bader2006factor}.
Applying Theorem \ref{thm:trivial fixed points in products}  with respect to each  factor $X_i$  individually implies that $F_i(\Gamma) = \mathrm{Fix}_{\partial X_i}(\Gamma) = \emptyset$ for each index $i\in\{1,\ldots,k\}$. We may now argue exactly as in the proof of Theorem \ref{thm:a discrete SRS has full limit set} with respect to the action of $\nu$-almost every subgroup $\Gamma$ on each factor $X_i$ via the projection map $p_i$. The only difference is that Proposition \ref{prop:a discrete SRS cannot have a single fixed point on the boundary} preventing the possibility of admitting a single fixed point on the boundary in the case of a single hyperbolic space is replaced by the   information obtained from  Theorem \ref{thm:trivial fixed points in products}.
\end{proof}

\subsection*{Product of $\CAT{-1}$-spaces}

Fix some $k \in \mathbb{N}$.  Let $X_1,\ldots,X_k$ be a collection of proper $\CAT{-1}$-spaces. 
Let $X$ denote the product space $X = X_1 \times \cdots \times X_n$ regarded with the product $\CAT{0}$ metric \cite[p. 168]{bridson2013metric}.

The boundary at infinity $\partial X$  is a $\CAT{1}$-space with respect to the angular metric \cite[Theorem II.9.13]{bridson2013metric}. It is isometric   to the spherical join of the individual boundaries \cite[\S II.8.11]{bridson2013metric}, i.e. 
\begin{equation}
\partial X = \partial X_1 \ast \partial X_2 \ast \cdots \ast  \partial X_k.
\end{equation}

The boundary $\partial X$ is realized as follows. Let $\Sigma$ be a topological realization of the combinatorial $(k-1)$-simplex with vertex set $V = \{1,\ldots,k\}$.  Then
\begin{equation}
\partial X = \partial X_1 \times \partial X_2 \times \cdots \times \partial X_k \times \Sigma / \thicksim
\end{equation}
where $\thicksim$ is the equivalence relation defined in terms of the quotient map taking  points lying over some face $F$ of the simplex $\Sigma$  to the sub-product of these factors $\partial X_i$'s for which $i \in F$. 
The fiber of the natural continuous map  $\sigma : \partial X \to \Sigma$ over every interior point of the simplex $\Sigma$ is homeomorphic to the direct product of the individual boundaries $\partial X_i$'s. More generally, the fiber of the map $\sigma$ over a point lying on the face $F$ of the simplex $\Sigma$ is homeomorphic to the product of these factors $X_i$ for which $i \in F$.
A boundary point $\zeta \in \partial X$ is called \emph{regular} if $\sigma(\zeta) \in \mathring{\Delta}$. Otherwise the boundary point $\zeta$ is called \emph{singular}.

% \begin{lemma}
% \label{lem:on fixed points sets}
% Let $\Gamma $ be a subgroup of $G$. Then
% \begin{equation}
% \mathrm{Fix}_{\partial }
% \end{equation}
% \end{lemma}

With the additional standing assumption that the spaces $X_i$ are $\CAT{-1}$ rather than Gromov hyperbolic, it is possible to  extend Theorem \ref{thm:trivial fixed points in products} with respect to all points at infinity of the product $\CAT{0}$-space $X$, regular as well as singular ones.

\begin{cor}
\label{cor:minimality for action on product of cat-1}
Let $\nu$ be a discrete $\mu$-stationary random subgroup of the product $\prod_{i=1}^k \Isom{X_i}$. If $\nu$-almost every subgroup projects non-trivially to each factor $\Isom{X_i}$ then $\nu$-almost every subgroup $\Gamma$ has $\mathrm{Fix}_{\partial X}(\Gamma) = \emptyset$.
\end{cor}
\begin{proof}
Note that the elliptic radical of $\Isom{Z}$ is trivial for any $\CAT{-1}$-space $Z$. For each discrete subgroup $\Gamma \in \DSub{G}$ denote
\begin{equation}
F_i(\Gamma) = \mathrm{Fix}_{\partial X_i}(\Gamma) \subset \partial X_i \quad \text{for each} \quad i\in\{1,\ldots,k\}.
\end{equation}
The spherical join
\begin{equation}
F(\Gamma) = F_1(\Gamma) \ast \cdots \ast F_k(\Gamma)
\end{equation}
can be naturally identified with a closed subset of the boundary $\partial X$. It can be seen from the above explicit realization of the boundary $\partial X$ that  $F(\Gamma) = \mathrm{Fix}_{\partial X}(\Gamma)$. In particular $F(\Gamma) = \emptyset$ if and only if $F_i(\Gamma) = \emptyset$ for each $i\in\{1,\ldots,k\}$. With this information at hand the corollary follows immediately from Theorem \ref{thm:trivial fixed points in products} applied with respect to each factor $X$ individually. We recall that the measure $\nu$ is in fact $\mu_i$-stationary for each $i$ by  \cite[Lemma 3.1]{bader2006factor}.
\end{proof}

\section{$\mathrm{CAT}(0)$-spaces and geometric density}
\label{sec:CAT0}

Let $X$ be a proper geodesically complete $\mathrm{CAT}(0)$-space with a fixed base point $x_0 \in X$. We assume that the group of isometries $\Isom{X}$ is acting on the space $X$ cocompactly and without fixed points at infinity. 

\begin{defn}
A subgroup $\Gamma$ of the isometry group $\Isom{X}$ is acting
\begin{itemize}
\item  \emph{minimally} if there are no proper closed convex $\Gamma$-invariant subsets, and
\item  \emph{geometrically densely} if it acts minimally on $X$ and without fixed points on the boundary $\partial X$. 
\end{itemize}
\end{defn}

Our standing assumptions imply that the space $X$ is \emph{boundary minimal}, i.e every  closed convex subset $Y \subsetneq X$ has $\partial Y \subsetneq \partial X$  \cite[Proposition 1.5]{caprace2009isometry}.
In particular a subgroup $\Gamma$ of $\Isom{X}$  acts geometrically densely provided it  admits no proper closed invariant subsets on the boundary $\partial X$. Some information about the converse direction of this statement is provided by the following result.

\begin{lemma}
\label{lem:no proper invariant subset CAT0}
Let $G$ be a subgroup of $\Isom{X}$ acting geometrically densely  on the space $X$ and  $\emptyset \subsetneq B \subset \partial X$ be a $G$-invariant closed subset. Then no proper closed convex subset $ Y \subsetneq X$ has $B \subset \partial Y$.
\end{lemma}
\begin{proof}
Assume towards contradiction that the family
\begin{equation}
\mathcal{F} = \{ \text{$Y \subsetneq X$ is a closed convex subset such that $B \subset \partial Y$} \}
\end{equation}
if non-empty. Consider the intersection
\begin{equation}
Z = \bigcap_{Y \in \mathcal{F}} Y
\end{equation}
so that $Z$ is proper $G$-invariant closed convex subset of $X$. The minimality of the $G$-action implies that  $Z = \emptyset$.  This means that the boundary subset $C = \bigcap_{Y \in \mathcal{F}} \partial Y$ has an intrinsic circumradius of at most $\pi/2$ with respect to the Tits angle\footnote{The Tits angle is called angular metric in \cite[II.9]{bridson2013metric}}  on the boundary $\partial X$ \cite[Proposition 3.2]{caprace2009isometry}. Additionally $B \subset C$ so that $C \neq \emptyset$. Therefore the subset $C$ has a unique circumcentre $ z \in \partial Y$   \cite[Proposition 3.2]{caprace2009isometry}. The  $G$-action fixes the point  $z$. This is a contradiction.
\end{proof}

\subsection*{Random walks and boundary measures}

Let $\mu $ be a probability measure on the group $\Isom{X}$. Assume that the semigroup $G_\mu$ generated  by the support of $\mu$ is a group acting geometrically densely on the space $X$. We further  assume that the measure $\mu$ has positive drift, in the sense that  the parameter $A$ given by
\begin{equation}
A = \lim_{n\to\infty} \frac{1}{n} \|\omega_n\| = \lim_{n\to\infty} \frac{1}{n} d_X(\omega_n x_0, x_0) 
\end{equation}
for $\mathrm{P}_\mu$-almost every sample path $\omega \in G^\mathbb{N}$ satisfies $A > 0$, see \cite[Theorem 2.1]{karlsson1999multiplicative}.

%Let $G$ a group of isometries of $X$ without fixed points on the boundary. Assume that the space $X$ is boundary minimal. Let $\mu$ be a probability measure on $G$ such that $\left<\mathrm{supp}(\mu)\right>$ acts minimally without fixed points on the boundary.
%Let $\nu$ be a $\mu$-stationary measure on $\mathrm{Sub}(G)$.

\begin{theorem}[Karlsson--Margulis \cite{karlsson1999multiplicative}]
For $\mathrm{P}_\mu$-almost every sample path $\omega \in G^\mathbb{N}$ the sequence $\omega_n x_0$ converges to a boundary point $\zeta_\omega \in \partial X$ depending on $\omega$.
\end{theorem}

Let $\nu_X$ be the probability measure  on the boundary $\partial X$ obtained by pushing forward the measure $\mathrm{P}_\mu$ via the map $\omega \mapsto \zeta_\omega$. The measure $\nu_X$ is $\mu$-stationary.

 %Let $\nu_X$ be the unique $\mu$-stationary measure on the visual boundary $\partial X$. Note  In particular $\nu_X=\pi_* \mathrm{\mu}^{\otimes \mathbb{N}}$. See the main result of \cite{karlsson1999multiplicative}.

%Denote $G=\left<\mathrm{supp} \mu\right>$. Assume $G$ is acting geometrically on the space $X$. The support is $G$-invariant.

The following result in the  context of general $\CAT{0}$-spaces  is to be compared with e.g. \cite[Proposition 9.1.b]{benoist2016random} in the context of symmetric spaces.

\begin{prop}
\label{prop:no proper support CAT0}
Any   proper closed convex  subset $Y \subsetneq X$ has $\mathrm{supp}(\nu_X) \not\subset \partial Y$.
\end{prop}
\begin{proof}
The probability measure $\nu_X$ is quasi-invariant by Lemma \ref{lem:Stationary measures are quasi-invariant}. In particular   $\mathrm{supp}(\nu_X)$ is a $G$-invariant subset of the boundary $\partial X$. The desired statement follows immediately from Lemma \ref{lem:no proper invariant subset CAT0}.
\end{proof}

\subsection*{Stationary random convex subsets}

Assume that the group $\Isom{X}$ is acting geometrically densely on the space $X$. We are ready to formulate a $\CAT{0}$-space variant of Proposition \ref{prop:no stationary measure on proper at infinity subsets}.

\begin{prop}
\label{prop:CAT(0) SRS cannot have minimal proper convex closed subsets}
Let $\nu$ be a $\mu$-stationary probability measure on $\Cl{X}$ such that $\nu$-almost every closed subset is convex. Then $\mathrm{supp}(\nu) \subset \{\emptyset, X\}$.
\end{prop}

\begin{proof}
It will be convenient to assume that the measure $\nu$ is ergodic. Say towards contradiction that $\nu$-almost every closed convex subset $Y$   has $ \emptyset \subsetneq Y \subsetneq X$. Consider the $\nu$-measurable  function
\begin{equation}
\label{eq:distance frm Y CAT0}
\Phi : \Cl{X} \setminus \{\emptyset\} \to \left[0,\infty\right), \quad\Phi : Y \mapsto d_X(x_0,Y).
\end{equation}
We know that  $\nu$-almost every closed convex subset $Y \subset X$ satisfies $\mathrm{supp}(\nu_X)\subsetneq \partial Y$, see Proposition \ref{prop:no proper support CAT0}. This means that a sample path $(\omega_n) \in G^\mathbb{N}$ with $\omega_n=g_1...g_n$ satisfies  $\omega_n x_0 \to \zeta$ for some $\zeta \notin \partial Y$ with positive $\mathrm{P}_\mu$-probability. In particular
\begin{equation}
\Phi(g_n^{-1}\cdots g_1^{-1} Y) = d_X(\omega_n x_0, Y) \xrightarrow{n\to\infty} \infty.
\end{equation}
This is a contradiction to Corollary \ref{cor:liminf of a Borel function is bounded} applied with respect to the function $\Phi$ defined in Equation (\ref{eq:distance frm Y CAT0}) and the probability measure $\check{\mu}$ given by $\check{\mu}(A) = \mu(A^{-1})$.
\end{proof}

\begin{cor}
\label{cor:CAT(0) no fixed points on the boundary}
Let $\nu$ be a $\mu$-stationary random subgroup of $\Isom{X}$. If $\nu$-almost every subgroup fixes no point of the boundary $\partial X$ then $\nu$-almost every subgroup is acting minimally, and is in  particular  geometrically dense.
\end{cor}
\begin{proof}
Our standing assumptions on the space $X$ imply that  its Tits boundary $\partial X$ is finite dimensional \cite{kleiner1999local}. This fact combined with the assumption that $\mu$-almost every subgroup has no fixed points on the boundary rules out case (A) of   the dichotomy presented in \cite[Theorem 4.3]{caprace2009isometry}. We deduce that $\nu$-almost every subgroup $H$ admits a \emph{canonical} minimal  closed convex non-empty invariant subset $\mathcal{M}(H) \in \Cl{X}$, see case (B.iii) of \cite[Theorem 4.3]{caprace2009isometry}.
Consider the $\mu$-stationary pushforward  probability measure $\mathcal{M}_* \nu$ on the space of closed subsets $\Cl{X}$. According to Proposition \ref{prop:CAT(0) SRS cannot have minimal proper convex closed subsets} it must be the case that $\mathcal{M}_* \nu = \mathrm{D}_X$. In other words $\nu$-almost every subgroup is acting minimally.
\end{proof}

As an application of our results on $\nu$-stationary subgroups of the isometry group $\Isom{X}$ we obtain the following. This can be seen as a stationary random subgroup  variant   of the geometric Borel density theorem, see \cite[Theorem 2.4]{caprace2009discrete} as well as the invariant random subgroup version  \cite{duchesne2015geometric}.

\begin{theorem}
\label{thm:discrete IRS of CAT are geometrically minimal} 
Assume that the $\CAT{0}$-space $X$ is given by $ X_1 \times \cdots \times X_k$ where each $X_i$ is a  proper $\CAT{-1}$-space with $\Isom{X_i}$ acting geometrically densely. Let $\nu$ be a discrete $\mu$-stationary  random subgroup of the product $\prod_{i=1}^k \Isom{X_i}$. If $\nu$-almost every subgroup projects non-trivially to each factor then $\nu$-almost every subgroup is acting geometrically densely.
\end{theorem}
\begin{proof}
We know that $\nu$-almost every subgroup has no fixed points on the boundary $\partial X$ by Corollary \ref{cor:minimality for action on product of cat-1}. Therefore $\nu$-almost every subgroup is geometrically dense by Corollary \ref{cor:CAT(0) no fixed points on the boundary}.
\end{proof}

Finally we have  the following $\CAT{0}$-space variant of Proposition \ref{prop:limit set cannot be empty}.

\begin{prop}
 Let $\nu$ be a  $\mu$-stationary random subgroup of $\Isom{X}$ such that $\nu$-almost every subgroup is non-trivial.  Then $\nu$-almost every subgroup   has unbounded orbits in $X$.
\end{prop}
\begin{proof}
Let $\Phi : \Sub{G} \to \Cl{X}$ be the $G$-equivariant $\nu$-measurable function taking a closed subgroup $H \in \Sub{G}$ to its closed convex fixed point set $\mathrm{Fix}(H) \in \Cl{X}$. Consider the $\mu$-stationary pushforward probability measure $\Phi_* \nu$ on the space of closed subsets $\Cl{X}$. It follows from Proposition \ref{prop:CAT(0) SRS cannot have minimal proper convex closed subsets} that $\Phi_* \nu = \alpha \mathrm{D}_\emptyset + (1-\alpha) \mathrm{D}_X$ for some real number $\alpha \in \left[0,1\right]$. The only subgroup $H \in \Sub{G}$ with $\Phi(H) = X$ is the trivial subgroup. On the other hand, recall that any bounded subset of a $\CAT{0}$-space has a unique circumcentre \cite[II.2]{bridson2013metric}. Therefore 
 any subgroup $H \in \Sub{G}$ having bounded orbits in $X$ satisfies   $\Phi(H) \neq \emptyset$. The desired conclusion follows by combining these facts.
\end{proof}

\section{Random walks and Poincare series}
\label{sec:random walks and poincare series}

 The main result of this section is Theorem \ref{thm:improvedtanaka}. It is stated below after we introduce some relevant notions concerning random walks.

Let $X$ be a proper unbounded metric space with a fixed base point $x_0 \in X.$
Let $\Gamma$ be a discrete group acting properly and cocompactly by isometries on the metric space $X$. Following our standing notation, we will write \begin{equation}
\|g\| = d_X(gx_0,x_0)
\end{equation}
 for every element $g \in \Gamma$, with the chosen base point $x_0$ implicit in the notation.

 \emph{We emphasize that   the space $X$ or (what is equivalent) the group $\Gamma$ are not required to be Gromov hyperbolic.}

%there are constant $c,K>0$  such that
%\begin{equation*}
%\mu( \{ g \in \Gamma \: : \: \|g\| >  r \} %)\leq c e^{-Kr}
%\end{equation*}
%for all $r > 0$.

\subsection*{Green's function}
Let $\mu$ be a  probability measure on the group $\Gamma$. Assume that 
\begin{itemize}
    \item the support of the measure $\mu$  generates the group $\Gamma$ as a semigroup,
    \item the  random walk determined by the measure $\mu$ is transient, and
    \item the measure $\mu$  has \emph{finite first moment}, namely  $\sum_{g\in \Gamma} \mu(g) \|g\| <\infty$. 
\end{itemize}
Note that the random walk determined by $\mu$ will be transient provided the   group $\Gamma$ is not virtually $\mathbb
{Z}$ or $\mathbb
{Z}^2$.
Consider the \emph{Green's function} associated to the probability measure $\mu$
\begin{equation}
\mathcal{G}(g,h)=\sum^\infty_{n=0}\mu^{*n}(g^{-1}h)
\end{equation}
 as well as the \emph{first return Green's function}
 \begin{equation}
 \mathcal{F}(g,h)=\frac{\mathcal{G}(g,h)}{\mathcal{G}(e,e)}\end{equation}
defined for all elements $g,h \in \Gamma$.
The transience assumption implies that the quantities $\mathcal{F}(g,h)$ and $\mathcal{G}(g,h)$ are  finite for any pair of elements $g,h\in \Gamma$.
 The first return Green's function satisfies
 \begin{equation}
 \mathcal{F}(g,h)=\mathrm{P}_\mu ( \{ \omega \: : \: \text{$g^{-1} h = \omega_n$ for some $n \in \mathbb{N}$} \} )
 \end{equation}
 for all elements $g,h \in \Gamma$.

 The expression $d_\mu=-\log \mathcal{F}$ defines a (possibly asymmetric) distance function on the group $\Gamma$ called the \emph{Green metric}.
The Green metric $d_\mu$ is quasi-isometric to a word metric on the group $\Gamma$ (see \cite[Proposition 3.6]{blachere2011harmonic} for symmetric measures and \cite[Proposition 7.8]{GekhtmanTiozzoGibbs} in general).

\begin{prop} 
If the group $\Gamma$ has exponential growth
%\footnote{Every non-amenable group has exponential growth.}
\label{subexponentialgrowthofgreenfunction} then the expression 
\begin{equation}
\sum_{\substack{g \in \Gamma \\  \|g\| \le n }} \mathcal{F}(e,g)
\end{equation}
 grows subexponentially in $n$.
\end{prop}
\begin{proof}
By the Harnack inequality   there are constants $k > 0$ and $ 0 <  t < \infty$ depending the group $\Gamma$ and the measure $\mu$ so that 
%(valid for any transient  random walk on a finitely generated group) 
\begin{equation}
    \mathcal{G}(g,h)\geq k t^{\|g^{-1}h\|}
\end{equation} 
for all elements $g,h \in \Gamma$, see e.g. \cite[(25.1)]{Woess2000}. In particular there is a constant $c > 0$ so that $d_\mu(e,g)\leq c\|g\|$ for all elements $g \in \Gamma$.
Therefore
\begin{equation}
    \sum_{\substack{g \in \Gamma \\  \|g\| \le n }} \mathcal{F}(e,g)
\leq \sum_{\substack{g \in \Gamma\\ d_\mu(e,g) \le c n} } e^{-d_\mu(e,g)}\leq b \sum^{\lceil c n \rceil }_{i=1}e^{-i} |\{g\in \Gamma:d_\mu(e,g)\leq i\}|
\end{equation}
for some constant $b>0$ which depends on the group $\Gamma$ and the measure $\mu$. The exponential growth condition implies  that 
$|\{g\in \Gamma:d_\mu(e,g)\leq i\}|\leq \varphi(i)e^i$ for some subexponentially growing function $\varphi : \mathbb{N} \to \mathbb{N}$  \cite[Proposition 3.1]{Blachereentropy}. This completes the proof.
\end{proof}

\subsection*{Entropy and drift}
Recall that $\mu$ is a probability measure on the  group $\Gamma$ such that $\mathrm{\supp}(\mu)$ generates $\Gamma$ as a semigroup. 

The \emph{drift} of the measure $\mu$ is 
\begin{equation}
\label{eq:define l}
 l(\mu) = \lim_{n\to\infty} \frac{1}{n} \|\omega_n\|.
 %\int_{\Isom{X}^\mathbb{N}} \|\omega_n\| \; \mathrm{d} \mu^{\otimes \mathbb{N}}(\omega).
\end{equation}
The first moment assumption implies that this limit exists for $\mathrm{P}_\mu$-almost every sample path $\omega$ and is independent of $\omega$. We have $l(\mu) > 0$  whenever the group $\Gamma$ is non-amenable, see e.g. \cite[Corollary 8.15]{Woess2000}.

The \emph{entropy} of the measure $\mu$ is 
\begin{equation}
\label{eq:define h}
 h(\mu) = \lim_{n\to\infty} \frac{1}{n} \mathrm{H} (\mu^{*n})
\end{equation}
where $\mathrm{H}(\cdot)$ denotes  Shannon entropy. 
This limit  exists  and is finite by the finite first moment condition. Again, we have $h(\mu) > 0$ whenever the group $\Gamma$ is non-amenable \cite{kaimanovich1983random}. 
The entropy $h(\mu)$ coincides with the drift with respect to the Green metric \cite[Theorem 1.1]{Blachereentropy}.
%In other words 
%\begin{equationh(\mu)=\lim_{i\to \infty} \frac{-\log \mathcal{F}(e,\omega_i)}{i}$ for $\mathrm{P}_\mu$ almost every sample path $\omega$.

 The  fundamental inequality of Guivarch relates entropy and drift \cite{Guivarch}. It states that 
 \begin{equation}
 h(\mu)\leq l(\mu) \delta(\Gamma)
 \end{equation}
 where $\delta(\Gamma) $ is the critical exponent\footnote{The notion of critical exponent is discussed in \S\ref{sec:critical exponents} below.} of the group $\Gamma$ with respect to its action on the metric space $X$. See also \cite{Vershik, Blachereentropy}.

From now on we will assume that the group $\Gamma$ is \emph{non-amenable} (so that in particular $h(\mu) > 0$ and $l(\mu) > 0$).
For our purposes it will be convenient to introduce the notation
\begin{equation}
\delta(\mu) = \frac{h(\mu)}{l(\mu)}.
\end{equation}

\subsection*{Annuli and hitting measures}
For each $i\in\mathbb{N}$   consider the \emph{annulus} given by
\begin{equation}
\label{eq:annuli}
A_i=\{g\in \Gamma: i- 1\leq \|g\| < i\}.
\end{equation}

Assume that $\mu$ is a probability measure on the infinite group $\Gamma$ with finite first moment.
Given a sample path $\omega \in \Gamma^\mathbb{N}$ and for each $i \in \mathbb{N}$ there is a   \emph{first hitting time} 
\begin{equation}
\tau_i(\omega) = \inf \{ n \in \mathbb{N} \: : \: \omega_n \in A_i \}.
\end{equation}
 The first hitting time $\tau_i(\omega)$ may be infinite\footnote{This is precisely the technical challenge that comes up when working with infinitely supported probability measures.} in the situation where  $\omega_n \notin A_i$ for all $n \in \mathbb{N}$. We define
 \[ \theta(\omega) = \{ i \in \mathbb{N} \: : \: \tau_i(\omega) < \infty \} \subset \mathbb{N}. \]
The positivity of the drift $l(\mu)$  implies  that the set $\theta(\omega)$ is infinite for $\mathrm{P}_\mu$-almost every sample path $\omega \in \Gamma^\mathbb{N}$.
 
Let $\nu_i$ be the  \emph{first hitting measure}  on the annulus $A_i$ determined at the  first hitting time $\tau_i$ for each $i\in \mathbb{N}$. Formally this means
\begin{equation}
\nu_i(E)=\mathrm{P}_\mu ( \{ \omega \: : \: i \in \theta(\omega) \quad \text{and} \quad \omega_{\tau_{i}(\omega)}\in E \})
\end{equation}
for each subset $E \subset A_i$. 
Formally $\nu_i$ is a positive measure on the annulus $A_i$ satisfying  $0 \le \nu_i(A_i) \le 1$.

\begin{lemma}
\label{lemma:from the blackboard}
There is a constant $  \zeta(\mu) > 0$ such that  $\mathrm{P}_\mu$-almost every sample path $\omega \in \Gamma^\mathbb{N}$ satisfies
\begin{equation}
\liminf_{n\to\infty} \frac{| \theta(\omega) \cap \left[n\right] |}{n} > \zeta(\mu) .
\end{equation}
\end{lemma}
\begin{proof}
For each given sample path $\omega \in \Gamma^\mathbb{N}$ and each $n \in \mathbb{N}$  consider the closed interval 
\begin{equation}
I_n(\omega) = \left[0, \max_{i \in \left[n\right]} \|\omega_i\| \right] \subset \mathbb{R}.
\end{equation}
Further, for each $K > 0$, let $L_{n,K}(\omega)$ denote the total length of the open subintervals of $I_n(\omega)$ of length $ \ge K$ not containing any of the 
$n$-many points $\{\|\omega_1\|,\ldots,\|\omega_n\|\}$. By the triangle inequality
\begin{equation}
| \|\omega_{i}\| - \|\omega_{i-1}\| | \le \|\omega_{i-1}^{-1} \omega_{i} \| = \|g_{i}\|
\end{equation}
where $(g_i) \in \Gamma^\mathbb{N}$ is the sequence of increments corresponding to the sample path $\omega$. It follows that $L_{n,K}(\omega)$ is upper bounded by the quantity $M_{n,K}(\omega)$ given by
\begin{equation}
M_{n,K}(\omega) = \sum_{ \{ i \in \left[n\right] \: : \: \|g_i \| \ge K \} } \| g_i \|.    
\end{equation}  
The strong law of large numbers implies for $\mathrm{P}_\mu$-almost every sample path $\omega \in \Gamma^\mathbb{N}$ that
\begin{equation}
\label{eq:ilyas f(K)}
   \lim_{n \to \infty} \frac{1}{n} M_{n,K}(\omega) = \sum_{\substack{g \in \Gamma\\ \|g\| \ge K}} \mu(g) \|g\|.
\end{equation}
Finally, the finite first moment assumption means that the right hand side in Equation (\ref{eq:ilyas f(K)}) tends to $0$ as $K$ tends to infinity. In follows that \begin{equation}
\lim_{K \to \infty} \limsup_{n \to \infty} \frac{1}{n} L_{n,K}(\omega) = 0.
    \end{equation}
On the other hand, recall that by definition $\lim_{n \to \infty} \frac{\|\omega_n\|}{n} = l(\mu) > 0$. Therefore, provided the parameter $K  > 0$ is sufficiently large, a positive proposition of the points in the interval $I_n(\omega)$ are within distance at most $K$ from some point of the form $\|\omega_i\|$. This is equivalent to the desired conclusion, for a suitable choice of the constant $\zeta(\mu) > 0$.
\end{proof}

\begin{notation}
$\lim_{i\in \theta(\omega)}$ is  a limit taken over all indices $i$ from the set $\theta(\omega)$ in ascending order. This makes sense for $\mathrm{P}_\mu$-almost every sample path $\omega$ where the subset $\theta(\omega)$ is infinite.
\end{notation}

The following  is stated inside  the proof of \cite[Theorem 6.1]{tanaka2017hausdorff} for finitely supported random walks on hyperbolic groups. For the sake of the completeness we include the proof, communicated to us by Tanaka. It  works identically in our setting.

\begin{lemma}\label{hittingrate}
$\mathrm{P}_\mu$-almost every sample path  $\omega   \in \Gamma^{\mathbb{N}}$ satisfies 
\begin{equation}
\lim_{i\in\theta(\omega)} \frac{\tau_i(\omega)}{i}  \to  \frac{1}{l(\mu)}.
\end{equation}
%where $\omega_n=g_1...g_n$
\end{lemma}
\begin{proof}
Clearly $\mathrm{P}_\mu$-almost every sample path $\omega$ is such  that the times $\tau_i(\omega) $ with $i \in \theta(\omega)$ are pairwise distinct. In particular   $\tau_i(\omega)\to \infty$ as $i \to \infty$. Moreover, $\mathrm{P}_\mu$-almost every sample path $\omega$ has 
\begin{equation}
\lim_{i\in\theta(\omega)} \frac{\| \omega_{\tau_i(\omega)} \|}{ \tau_{i}(\omega)} \to l(\mu)
\end{equation}
by definition of the drift $l(\mu)$. Furthermore $\| \omega_{\tau_i(\omega)} \| \in \left[i-1,i\right)$ for all $i\in\mathbb{N}$ by definition of the first hitting time $\tau_i$. The result follows.
\end{proof}

\subsection*{An estimate on first hitting measures}
Roughly speaking, we show that the first hitting measure $\nu_i$ of \enquote{most points} in the $i$-th annulus $A_i$ decays as $e^{-i \delta(\mu)}$. Here is the precise statement.

\begin{theorem}  
\label{thm:Tanaka}
For every $a>0$ we have
\begin{equation}
\lim_{i\to \infty} \nu_{i} \left(  \Gamma \setminus \{ g\in A_i: e^{-\delta(\mu)(i+a)} \leq \nu_{i}(\{g\})\leq e^{-\delta(\mu)(i-a)} \} \right)=0.
\end{equation}
\end{theorem}

The special case of Theorem \ref{thm:Tanaka}  where the measure $\mu$ is finitely supported and the group in question $\Gamma$ is hyperbolic   is Tanaka's \cite[Theorem 6.1.2]{tanaka2017hausdorff}. We provide a detailed proof in the infinitely supported non-amenable case by using our  Proposition \ref{subexponentialgrowthofgreenfunction}. 
In any case,  Theorem \ref{thm:Tanaka} follows immediately by integrating the statement of the following  Proposition \ref{Tanaka6.1a-infsupport} with respect to the measure $\mathrm{P}_\mu$ on the space of sample paths.

\begin{prop}\label{Tanaka6.1a-infsupport}
$\mathrm{P}_\mu$-almost every sample path $\omega   \in \Gamma^{\mathbb{N}}$ satisfies 
\begin{equation}
\label{eq:to be established in both directions}
\lim_{i\in \theta(\omega)} -\frac{\log \nu_i(\{\omega_{\tau_i(\omega)}\})}{i}  = \delta(\mu).
\end{equation}
\end{prop}
\begin{proof}
Our proof is inspired by Tanaka's argument in \cite[Theorem 6.1]{tanaka2017hausdorff}. Observe that using
 first hitting measures which are not  probability measures as well as the fact that the sum of the Green function over an annulus grows subexponentially instead of being bounded does not complicate things much.

We argue by establishing an inequality in both directions of Equation (\ref{eq:to be established in both directions}). We have already mentioned  the  fact   that  the entropy $h(\mu)$ is equal to the drift with respect to the Green metric \cite[Theorem 1.1]{Blachereentropy}. By definition, this means that $\mathrm{P}_\mu$-almost every sample path $\omega$ satisfies
\begin{equation}
\label{eq:limit for one direction}
\lim_{n\to\infty} -\frac{\log \mathcal{F}(e,\omega_n)}{n} = h(\mu).
\end{equation}
By the definition of the first return Green's function and the first hitting time measure it is immediate that
\begin{equation}
\label{eq:trivial bound F nu}
\nu_i(\omega_{\tau_i(\omega)})\leq \mathcal{F}(e,\omega_{\tau_i(\omega)})
\end{equation}
$\mathrm{P}_\mu$-almost surely and for all $i\in\theta(\omega)$.
Moreover $\tau_i (\omega)/i\to 1/l(\mu)$ for $i\in\theta(\omega)$ by Lemma \ref{hittingrate}. Equations (\ref{eq:limit for one direction}) and (\ref{eq:trivial bound F nu}) together imply that  $\mathrm{P}_\mu$-almost every sample path $\omega$ has 
\begin{equation}
\liminf_{i \in \theta(\omega)} - \frac{\log \nu_i(\{\omega_{\tau_i(\omega)}\})}{i} \ge \delta(\mu).
\end{equation}

We now establish the inequality in the other direction. For every choice of the two parameters  $s>\delta(\mu)$ and $0<t<(s-\delta(\mu))/2$ consider the set 
\begin{equation}
P_i=\{g\in A_i:\nu_i(g)\leq e^{-i(s-t)} \quad \text{and} \quad \mathcal{F}(e,g)\geq e^{-i(\delta(\mu)+t)}\}.
\end{equation}
Proposition \ref{subexponentialgrowthofgreenfunction} implies that
 \begin{equation}
 \sum_{g\in A_i} \mathcal{F}(e,g) \le  \varphi(i) 
 \end{equation}
 for some subexponentially growing function $\varphi : \mathbb{N} \to \mathbb{N}$. 
 We deduce that
 \begin{equation}
 |P_i|\leq  \varphi(i) e^{i(\delta(\mu)+t)}
 \end{equation}
 and so
 \begin{equation}
 \nu_i (P_i)\leq  \varphi(i) e^{i(\delta(\mu)+2t - s)}
 \end{equation}
 for all $i \in \mathbb{N}$.
The Borel--Cantelli lemma implies that for  $\mathrm{P}_\mu$-almost every sample path $\omega$   there exists an $N_\omega \in \mathbb{N}$  such
that  $\omega_{\tau_i(\omega)} \notin  P_i$  for all $i\geq N_\omega$ with $i \in \theta(\omega)$. On the other hand,  Equation (\ref{eq:limit for one direction})  combined with the fact that entropy is the drift with respect to the Green metric  implies that  $\mathrm{P}_\mu$-almost every sample path $\omega$  satisfies 
\begin{equation}
\mathcal{F}(e,\omega_{\tau_{i}(\omega)})\geq e^{-h(\mu)\tau_{i}(\omega)-\frac{t}{2}}\geq e^{-i(\delta(\mu)+t)}
\end{equation}
for all sufficiently large $i \in \theta(\omega)$.
The above facts combined with the definition of the sets $P_i$ imply that
\begin{equation}
\limsup_{i\in\theta(\omega)} -\frac{ \log \nu_i(\{\omega_{\tau_i(\omega)}\})}{i}<s-t
\end{equation} for $\mathrm{P}_\mu$-almost every sample path  $\omega$. Taking the parameter $s$ to be arbitrary close to $\delta(\mu)$ and the parameter $t$ to be arbitrary small completes the proof.
\end{proof}

The following  immediate corollary of Theorem \ref{thm:Tanaka}  can be regarded as a sharpened form of \cite[Theorem 6.1.3]{tanaka2017hausdorff}.

\begin{cor}\label{cor:Tanakaimproved}
For every $\varepsilon>0$ there are  constants $C>0$ and $N>0$ such that provided $i >N$ any subset $E \subset A_i$ with $\nu_{i}(E)\geq \varepsilon$ has cardinality $|E| \ge C e^{i \delta (\mu)}$.
\end{cor}

\subsection*{Divergence of Poincare series}
The following theorem is the main result of \S\ref{sec:random walks and poincare series}. Roughly speaking, it says that if a certain subset $W$ of the group $\Gamma$ is \enquote{large} in the statistical sense of being  frequently visited by the random walk, then it must also be \enquote{large} in the sense that the  partial Poincare series evaluated over $W$ diverges.

\begin{thm}
\label{thm:improvedtanaka}
Assume that the group $\Gamma$ is non-amenable and that the probability measure $\mu$ whose supported generates $\Gamma$ as a semigroup has finite first moment, i.e. 
\begin{equation}
\sum_{g\in \Gamma} \mu(g) \|g\| <\infty.    
\end{equation} 
Let $W\subset \Gamma$ be a subset. If $\mathrm{P}_\mu$-almost every sample path $\omega $ satisfies
\begin{equation}
\label{eq:Tanaka long visiting times}
\liminf_{n\to\infty} \frac{|\{i\in \left[n\right]  \: : \: \omega_i\in W\}|}{n}> 1 - l(\mu) \zeta(\mu)
\end{equation}
where  $\zeta(\mu) > 0$ is as in Lemma \ref{lemma:from the blackboard} then
\begin{equation}
\label{eq:partial Poincare with exp delta(mu)}
\sum_{\gamma \in W} e^{-\delta(\mu) \|\gamma\|} = \infty.
\end{equation}
\end{thm}

%The statement of Theorem \ref{thm:improvedtanaka} is a vast generalization of Tanaka's results in Section 6 of \cite{tanaka2017hausdorff}. Our result differs from Tanaka's in the following aspects
%While the arguments of \S\ref{sec:random walks and poincare series} are influenced by Tanaka's work \cite{tanaka2017hausdorff}, we   differ from Tanaka in several aspects. 

%First, we do not require the group $\Gamma$ (or the space $X$ it acts on) to be hyperbolic, and instead work  with arbitrarily cocompact actions of non-amenable groups. Furthermore, we assume that the measure $\mu$ has finite first moment rather than finite support. In addition, we assume that sample paths enter the subset $W$ statistically, rather than deterministically. Lastly, the divergence of the series in Equation (\ref{eq:partial Poincare with exp delta(mu)}) with exponent $\delta(\mu)$  is a stronger statement than the rate of growth estimate in \cite{tanaka2017hausdorff}.
 
%We are ready to establish the relationship between random walks and the divergence of a Poincare series.
We begin with the following Lemma.
\begin{lemma}\label{lemma:densityofhittingtimes}
$\mathrm{P}_\mu$-almost every sample path $\omega   \in \Gamma^{\mathbb{N}}$ satisfies 
\begin{equation} \label{eqn:densityofhittingtimes}
\liminf_{n\to\infty} \frac{|\{i \in \left[n\right] \: : \: \text{$i = \tau_j(\omega)$ for some $j\in \theta(\omega)$} \} |} {n} \geq l(\mu)\zeta(\mu) > 0
\end{equation}
    where $\zeta(\mu)$ is the constant given in Lemma \ref{lemma:from the blackboard}.
\end{lemma}
\begin{proof}
We know from Lemma \ref{lemma:from the blackboard} that for $\mathrm{P}_\mu$-almost every sample path $\omega \in \Gamma^\mathbb{N}$ the set $\theta(\omega)$ is infinite and satisfies
\begin{equation}
\liminf_{n\to\infty} \frac{| \theta(\omega) \cap \left[n\right] |}{n} > \zeta(\mu).    
\end{equation}
Furthermore,     Lemma
    \ref{hittingrate} says that $\mathrm{P}_\mu$-almost every sample path $\omega \in \Gamma^\mathbb{N}$ satisfies $\frac{\tau_{i}(\omega)}{i}\to \frac{1}{l(\mu)}.$
    Consider   a particular sample path $\omega \in \Gamma^\mathbb{N}$ satisfying both of these two conditions. In particular, when $n$ is large enough, we have
    \begin{equation}
        |\theta(\omega) \cap \left[ \left\lfloor nl(\mu)\right\rfloor \right] | \ge \zeta(\mu)\left\lfloor nl(\mu)\right\rfloor.
    \end{equation}
 Fix some small $\varepsilon>0$. There is some $N_\varepsilon > 0$ such that 
     $\tau_{n}(\omega)\leq (\frac{1}{l(\mu)} + \varepsilon) n$ for all $ n \ge N_\varepsilon$. Thus $\tau_{j}(\omega)\leq (1+\varepsilon l(\mu))n$ holds true for all $j \in \theta(\omega)$   such  that $N_\varepsilon \le j\leq \left \lfloor{n l(\mu)}\right \rfloor$.
        By the definition of the annuli $A_i$ in Equation (\ref{eq:annuli}), they are pairwise disjoint. Hence the hitting times $\tau_j(\omega)$ for $j\in \theta(\omega)$ are all pairwise distinct. 
    We obtain that
    \begin{equation}
    |\{ 1 \le i\leq (1+\varepsilon l(\mu))n \: : \:  \text{$i = \tau_j(\omega)$ for some $j\in \theta(\omega)$} \}|\geq \zeta(\mu)  \left\lfloor{n l(\mu)}\right \rfloor - N_\varepsilon
    \end{equation}
    for all $n$ sufficiently large.
    Taking $\varepsilon \to 0$ completes the proof of the lemma.
\end{proof}

\begin{proof}[Proof of Theorem \ref{thm:improvedtanaka}]
We know by Lemma \ref{lemma:densityofhittingtimes} that
 $\mathrm{P}_\mu$-almost every sample path $\omega \in \Gamma^\mathbb{N}$ satisfies 
\begin{equation} \label{densityofhittingtimes}
\liminf_{n\to\infty} \frac{|\{i \in \left[n\right] \: : \: \text{$i = \tau_j(\omega)$ for some $j$} \} |} {n} \geq l(\mu)\zeta(\mu) > 0.
\end{equation}

Let $W \subset \Gamma$ be a subset taken so that Equation (\ref{eq:Tanaka long visiting times}) holds. 
Combining Equations (\ref{densityofhittingtimes}) and (\ref{eq:Tanaka long visiting times}) we obtain that: \begin{equation}
 \varepsilon = \liminf_{n\to\infty} \frac{1}{n}  \sum_{i \in \theta(\omega) \cap \left[n\right]} 1_{W}(\omega_{\tau_i(\omega)}) > 0
 \end{equation}
 for $\mathrm{P}_\mu$-almost every sample path $\omega$. The definition of the first hitting measures  $\nu_i$ gives that
 \begin{equation}
 \label{eq:by the proof}
\liminf_{n\to\infty} \frac{ | \{ i \in \left[n\right] \: : \: \nu_i(W)>\varepsilon \} |}{n} >0
\end{equation}
 for $\mathrm{P}_\mu$-almost every sample path $\omega$. By Corollary \ref{cor:Tanakaimproved} this implies
 \begin{equation}
\liminf_{n\to\infty} \frac{|\{i \in \left[n\right] \: : \: 
|W \cap A_i|>C e^{i \delta(\mu)} \}|}{n}>0
\end{equation}
for some constant $C > 0$. 
Therefore 
\begin{equation}
\liminf_{n\to\infty} \frac{1}{n} \sum^{n}_{i=1} e^{-i \delta(\mu)} |W\cap A_i|>0.
\end{equation}
So
\begin{equation}
\sum^\infty_{i=0} e^{-i\delta(\mu)  } |W\cap A_i| = \infty.
\end{equation}
This equation is equivalent to the desired conclusion.
\end{proof}

\subsection*{Divergence of  Poincare series weighted by a function}

Lastly we consider  a certain modification of Theorem \ref{thm:improvedtanaka} allowing for a weight function.  It will be used below to study discrete stationary random subgroups of divergence type, see e.g. Proposition \ref{prop:KingmannforPScocycle} and the discussion surrounding it.

\begin{prop} 
\label{prop:TanakawithF}
Assume that the group $\Gamma$ is non-amenable. Let
\begin{itemize}
    \item $\mu$ be a probability measure on the group $\Gamma$ with finite first moment whose support generates $\Gamma$ as a semigroup.
\item $S:\Gamma \to \mathbb{R}_{>0}$ be a function satisfying $\lim_{n\to \infty}S(\omega_n)^{\frac{1}{n}}\to 1$ for $\mathrm{P}_\mu$-almost every sample path $\omega$.
\item  $W\subset \Gamma$ be a subset such that $\mathrm{P}_\mu$-almost every sample path $\omega$ has
\begin{equation}
\liminf_{n\to\infty} \frac{|\{i\in \left[n\right]  \: : \: \omega_i\in W\}|}{n}> 1 - \zeta
\end{equation}
where $\zeta > 0$ is the constant as in Lemma \ref{lemma:from the blackboard}.
\end{itemize}
Then for any $a>0$ we have that
\begin{equation}
\limsup_{i\to \infty}e^{-(\delta(\mu)-a)i}\sum_{g\in W \cap A_i}S(g)>0.
\end{equation}
\end{prop}

\begin{proof} 
It follows from Lemma \ref{hittingrate} that $\mathrm{P}_\mu$-almost every sample path $\omega$ satisfies
\begin{equation}
\lim_{i\in\theta(\omega)} S(\omega_{\tau_i(\omega)})^{1/i}=\lim_{i\in\theta(\omega)} 
 \left(S(\omega_{\tau_i(\omega)})^{\frac{1}{\tau_i(\omega)}}\right)^{\frac{\tau_i(\omega)}{i}} =  1.
\end{equation}
Together with Theorem \ref{thm:Tanaka} this  implies that for every   $a>0$ we have
\begin{equation}
\label{eq:S second last}
\lim_{i\to \infty} \nu_{i} \left( \Gamma \setminus  \{ g\in A_i \: : 
 \: e^{\delta(\mu)(i-a)} \leq \nu_{i}(\{g\})\leq e^{\delta(\mu)(i+a)} \; \text{and} \; e^{-ai}\leq S(g)\leq e^{ai} \} \right)=0.
\end{equation}
Consequently  for every choice of $a > 0$ and $ \varepsilon>0$ there are some constants $C>0$ and $N\in\mathbb{N}$ such that whenever   $i>N$  any subset $W\subset A_i$ with $\nu_{i}(W)\geq \varepsilon$ has  
\begin{equation}
\label{eq:S last}
\sum_{g\in W \cap A_i}S(g) \ge c e^{i (\delta (\mu)-a)}.
\end{equation}
By  Equation (\ref{eq:by the proof}) appearing in  the proof of Theorem \ref{thm:improvedtanaka} we have \begin{equation}
\label{equation:super last}
\liminf_{n\to\infty} \frac{ | \{ i \in \left[n\right] \: : \: \nu_i(W)>\varepsilon \} |}{n} >0
\end{equation}
 for $\mathrm{P}_\mu$-almost every sample path $\omega$. 
 The desired conclusion follows by putting together Equations and (\ref{eq:S last}) and (\ref{equation:super last}).
\end{proof}

\section{Critical exponents of stationary random subgroups}
\label{sec:critical exponents}

Let $X$ be a proper  Gromov hyperbolic geodesic metric space with a fixed base point $x_0 \in X$. Assume that the group $\Isom{X}$ is non-elementary. For every element $g \in \Isom{X}$ denote \begin{equation}
\|g\| = d_X(gx_0,x_0)
\end{equation}
with the chosen base point $x_0$ implicit in the notation. 

\begin{defn}
\label{def:Poincare series and critical exponent}
Given a discrete subgroup  $\Gamma$ of the group of isometries $\Isom{X}$:
\begin{itemize}
\item The \emph{Poincare series} of the subgroup $\Gamma$  at the exponent $s \ge 0$  is
\begin{equation*}
\mathcal{P}_{\Gamma}(s) = \sum_{\gamma \in \Gamma} e^{-s\|\gamma\|}.
\end{equation*}
\item The \emph{critical exponent} of the subgroup $\Gamma$ is the number
\begin{equation*}
\delta(\Gamma) = \sup \{ s \ge 0 \: : \: \mathcal{P}_\Gamma(s) < \infty \} = \inf \{ s \ge 0 \: : \: \mathcal{P}_\Gamma(s) = \infty \}.
\end{equation*}
\item The subgroup $\Gamma$ is of \emph{divergence type} if $\mathcal{P}_\Gamma(\delta(\Gamma)) = \infty$. \item The subgroup $\Gamma$ is of \emph{convergence type} if $\mathcal{P}_\Gamma(\delta(\Gamma)) < \infty$.
\end{itemize}
\end{defn}

 The following is the main result of this work.

\begin{theorem}
\label{thm:critical exponent for SRS}
Let $\mu$ be a  probability measure on the group $\Isom{X}$. Assume that
\begin{itemize}
    \item the measure $\mu$ has finite first moment, and
    \item the semigroup generated by the support of $\mu$ is  a discrete subgroup acting properly and cocompactly on the metric space $X$. 
\end{itemize}
Let $\nu$ be a discrete $\mu$-stationary random subgroup of $\Isom{X}$. If $\nu$-almost every subgroup is not contained in the elliptic radical $\mathrm{E}(\Isom{X})$ then the strict lower bound
\begin{equation}
\delta(\Delta) > \frac{\delta(\mu)}{2}
\end{equation}
holds for $\nu$-almost every discrete subgroup $\Delta$. 
 In particular
\begin{equation}
\mathcal{P}_\Delta\left(\frac{\delta(\mu)}{2} \right) = \infty
\end{equation}
holds $\nu$-almost surely.  
 Moreover if $\nu$-almost every subgroup is of divergence type then $\delta(\Delta) \ge \delta(\mu)$ holds true $\nu$-almost surely.
\end{theorem}

The remainder of this section is dedicated to giving a proof of Theorem \ref{thm:critical exponent for SRS}. Note that the assumption that the group $\Isom{X}$ contains a uniform lattice implies that it is of general type.

\subsection*{Divergence of Poincare series}

The following   uses the core   arguments of \cite{gekhtman2019critical} adapted to work in the context of random walks. The idea is to relate the critical exponents of $\nu$-almost every discrete subgroup with the quantity $\frac{\delta(\mu)}{2}$.

\begin{prop}
\label{prop:a condition for poincare series to diverge at half crit}
Let $V$ be an open and relatively compact subset of $\Isom{X}$ consisting of loxodromic elements. Let $\Delta \in \DSub{\Isom{X}}$ be a discrete subgroup. If
\begin{equation}
\liminf_{n \to \infty} \frac{|\{i\in\left[n\right] \: : \:  \Delta^{\omega_i} \cap V \neq \emptyset\} | }{n} > 1- \zeta(\mu)
\end{equation}
for $\mathrm{P}_\mu$-almost every sample path $\omega$ where  $\zeta(\mu) > 0$ is as provided by Lemma \ref{lemma:from the blackboard} then  $\mathcal{P}_{\Delta}(\frac{\delta(\mu)}{2}) = \infty$.
\end{prop}

\begin{proof}
To simplify notations write $\delta = \delta(\mu)$. Let $\Gamma$ be the uniform lattice in the group $\Isom{X}$ generated by the support of the measure $\mu$. Write
\begin{equation}
\mathcal{P}_\Delta(\delta/2) = \sum_{h \in \Delta} e^{-\frac{\delta}{2}   \|h\|}  \quad \text{so that} \quad \mathcal{P}_\Delta(\delta/2) \ge  \sum_{\substack{h \in \Delta \\ h^\Gamma \cap V \neq \emptyset}} e^{-\frac{\delta}{2}  \|h\|}.
\end{equation}
It was shown in    \cite[Lemma 3.3]{gekhtman2019critical} that there is a constant $\beta > 0$ depending only on the space $X$ and the subset $V$ such that
\begin{equation}
e^{-\frac{\delta}{2} \|h\| } \ge \frac{1}{\beta} e^{- \delta \|\gamma\| }
\end{equation}
for any element $\gamma \in \Gamma$ with $h^\gamma \in V$. Likewise,  according to \cite[Proposition 3.5]{gekhtman2019critical} there is a constant $\alpha > 0$ depending only on the space $X$ and the subset $V$ such that, if $h^\Gamma \cap V \neq \emptyset$ then there is an element $\gamma_h \in \Gamma$ with
\begin{equation}
e^{- \delta \|\gamma_h \| } \ge \frac{1}{\alpha} \sum_{\substack{\gamma \in \Gamma \\ h^\gamma \in V}} e^{-\delta \|\gamma\| }.
\end{equation}
Putting the above three equations together gives
\begin{equation}
\mathcal{P}_\Delta(\delta/2) \ge 
 \frac{1}{\beta} \sum_{\substack{h \in \Delta \\ h^\Gamma \cap V \neq \emptyset}} e^{- \delta  \|\gamma_h\|} \ge \frac{1}{\alpha \beta} \sum_{h \in \Delta} \sum_{\substack{\gamma \in \Gamma \\ h^\gamma \in V}} e^{-\delta \|\gamma\| } \ge \frac{1}{\alpha \beta} \sum_{\substack{\gamma \in \Gamma \\ \Delta^\gamma \cap V \neq \emptyset}} e^{-\delta \|\gamma\| }.
\end{equation}
To conclude apply Theorem \ref{thm:improvedtanaka} with the subset 
$
W = \{ \gamma \in \Gamma \: : \: \Delta^\gamma \cap V \neq \emptyset \}$.
\end{proof}

Our strategy towards Theorem \ref{thm:critical exponent for SRS} is to start by showing divergence at the exponent $\frac{\delta(\mu)}{2}$.

\begin{proof}[Proof of $\mathcal{P}_\Delta(\delta(\mu)/2) = \infty$ part in Theorem \ref{thm:critical exponent for SRS}]
Let $\nu$ be a discrete $\mu$-stationary subgroup not contained in the elliptic radical $\mathrm{E}(\Isom{X})$.
The action of the subgroup generated by the support of the measure $\mu$ has general type. We know from Corollary \ref{cor:SRS of general type are general type} that $\nu$-almost every subgroup is of general type and in particular admits loxodromic elements.
By exhaustion it is possible to choose a sufficiently large open and relatively compact subset $V$  of $\Isom{X}$ consisting of loxodromic elements so that
\begin{equation}
\nu(\{\Delta \in \DSub{\Isom{X}} \: : \: \Delta \cap V \neq \emptyset \}) > 1-\zeta(\mu)
\end{equation}
where  $\zeta(\mu) > 0$ is as provided by Lemma \ref{lemma:from the blackboard}.

We know from Corollary \ref{cor:liminf of visits to positive measure set} of Kakutani's ergodic theorem that $\mathrm{P}_\mu$-almost every sample path $\omega$ and $\nu$-almost every discrete subgroup $\Delta$ satisfy
\begin{equation}
\liminf_{n \to \infty} \frac{|\{i\in\left[n\right] \: : \:  \Delta^{\omega_i} \cap V \neq \emptyset\} | }{n} > 1- \zeta(\mu).
\end{equation}
At this point we conclude that $\mathcal{P}_\Delta(\delta(\mu)/2) = \infty$ by 
  Proposition \ref{prop:a condition for poincare series to diverge at half crit}.
\end{proof}

By regarding an invariant random subgroup as a stationary one it is now possible to recover the non-strict inequality  in the main result of \cite{gekhtman2019critical}.

\begin{cor}
\label{cor:ussing gmm to get one half}
Assume that the group $\Isom{X}$ admits a discrete subgroup acting cocompactly on the space $X$.
Let $\nu$ be a discrete invariant  random subgroup of $\Isom{X}$ not contained in the elliptic radical $\mathrm{E}(\Isom{X})$. Then $\nu$-almost every discrete subgroup $\Delta$ has $\delta(\Delta) \ge \frac{1}{2}\mathrm{dim}(\partial X)$. 
\end{cor}
\begin{proof}
Fix any uniform lattice $\Gamma$ in the group $\Isom{X}$. The construction of \cite[Theorem 1.4]{gouezel2018entropy} exhibits a  sequence $\mu_i$ of probability measures supported on finite subsets of the lattice $\Gamma$  and satisfying
$\delta(\mu_i) \to \frac{1}{2}\dim_\mathrm{Haus}(\partial X)$.
The desired conclusion now follows from Theorem \ref{thm:critical exponent for SRS}.
\end{proof}

\subsection*{Stationary random subgroups of divergence type}

 Let $\Gamma$ be a discrete subgroup  of   $\Isom{X}$ having divergence type. There exists  a $\Gamma$-\emph{quasi-conformal density} $\eta^\Gamma$ of dimension $\delta(\Gamma)$. By definition, this is a map assigning to every point $x \in X$   a positive measure $\eta_x^\Gamma$ on the boundary $\partial X$ satisfying certain axioms. Such a map is unique up to a multiplicative constant. See e.g.  \cite{nicholls1989ergodic}, \cite{coornaert1993mesures} or \cite[Definition 6.6]{gekhtman2019critical} for more detailed information.

Normalize  $\eta^\Gamma_{x_0}$ to be a probability measure. 
For discrete subgroups of divergence type, the total measure $\|\eta_x^\Gamma\| = \eta_x^\Gamma(\partial X)$ can be characterized in terms of a certain Poincare series. Namely, for each pair of points $x,y \in X$ denote
\begin{equation}
\mathcal{P}_\Gamma(s;x,y)=\sum_{\gamma \in \Gamma} e^{-sd(x,\gamma y)}.
\end{equation}
We get  as a special case $\mathcal{P}_\Gamma(s) = \mathcal{P}_\Gamma(s;x_0,x_0)$.
%in terms of the notation of Definition \ref{def:Poincare series and critical exponent},   where $x_0 \in X$ is our fixed base point. 
Note  that $\mathcal{P}_\Gamma(s;x,y)=\mathcal{P}_\Gamma(s;y,x)$. 
 For any pair of points $x,y \in X$ we have according to   \cite[Lemmas 6.1 and 6.2]{matsuzaki2020normalizer} that 
\begin{equation}
\label{eq:MYJ}
\|\eta^\Gamma_{x}\| \approx \lim_{s \searrow \delta(\Gamma)}\frac{\mathcal{P}_\Gamma(s;x,y)}{\mathcal{P}_\Gamma(s;x_0,y)}.
\end{equation}

\emph{
We emphasize at this point that the multiplicative constant implicit in all our formulas involving the notation $\approx$   depends only on the metric space $X$, and not on anything else.}

Let $\mathcal{M}(\partial X)$ denote the space of all positive Borel measures on the boundary at infinity $\partial X$. In fact, if $\nu$ is a discrete $\mu$-stationary random subgroup of divergence type then there is a $\nu$-measurable mapping
\begin{equation}
\eta : \DSub{\Isom{X}} \times X \to \mathcal{M}(\partial X), \quad \eta : (\Gamma,x) \mapsto \eta^\Gamma_x 
\end{equation}
defined for all discrete subgroups $\Gamma$ and all points $x \in X$  in such  a way  that the assignment $x \mapsto \eta^\Gamma_{x}$ is a $\Gamma$-quasi-conformal density of dimension $\delta(\Gamma)$.
Consider the expression
\begin{equation}
\label{eq:pi_nu}
\pi_\nu : \Isom{X} \times \DSub{\Isom{X}} \to \mathbb{R}_{> 0 }, \quad \pi_\nu(g,\Gamma) = \|\eta_{g^{-1}x_0}^\Gamma\| 
\end{equation}
for all elements $g \in  \Isom{X}$ and all subgroups $ \Gamma\in\DSub{\Isom{X}}$.
The function $\pi_\nu$ is a \emph{multiplicative quasi-cocycle}, i.e.
\begin{equation}
\pi_{\nu}(gh,\Gamma) \approx \pi_{\nu}(g,\Gamma^h) \pi_{\nu}(h,\Gamma).
\end{equation}
For all this see   \cite[\S5 and \S6]{gekhtman2019critical}.

By combining the definition of the quasi-cocycle $\pi_\nu$ in Equation (\ref{eq:pi_nu}) with the characterization of the total mass of the conformal density in Equation (\ref{eq:MYJ}) we obtain the formula
\begin{equation}
\pi_{\nu}(g,\Gamma)\approx \lim_{s\searrow \delta(\nu)} \frac{\mathcal{P}_\Gamma(s;g^{-1}x_0,g^{-1}x_0)}{\mathcal{P}_\Gamma(s;g^{-1}x_0,x_0)} \approx \lim_{s\searrow \delta(\nu)} \frac{\mathcal{P}_\Gamma(s;g^{-1}x_0,x_0)}{\mathcal{P}_\Gamma(s;x_0,x_0)}.
\end{equation}
One consequence of the above is that 
\begin{equation} \label{quasiPoincare}
 \pi_{\nu}(g,\Gamma)^2 \approx  \lim_{s\searrow \delta(\nu)} \frac{\mathcal{P}_\Gamma(s;g^{-1} x_0,g ^{-1} x_0)}{\mathcal{P}_\Gamma(s;x_0,x_0) } = \lim_{s\searrow \delta(\nu)} \frac{\mathcal{P}_{  \Gamma ^g}(s)}{\mathcal{P}_\Gamma(s)}.
\end{equation}

The triangle inequality gives for each 
  $s>\delta(\Gamma)$  that
  \begin{equation} \label{PSisL1}
    -2s\|g\| \leq \log\frac{\mathcal{P}_\Gamma(s;g^{-1}x_0,g^{-1}x_0)}{\mathcal{P}_\Gamma(s;g^{-1}x_0,x_0)}  \leq 2s\|g\|.
\end{equation}
In particular  $|\pi_\nu(g,\Gamma)|\leq 2 \delta(\Gamma) \|g\|$ for all $g \in \Isom{X}$ and $\Gamma \in \mathrm{DSub}(\Isom{X})$.

\begin{remark}
We have allowed for  a slight abuse of notation; the above limits  as $s \searrow \delta(\Gamma)$  might not  exist in the strict sense. The correct statement is that the limits only exist along some subsequence, and the limit is independent of the choice of this subsequence up to a uniform multiplicative constant.
\end{remark}

\begin{prop}\label{prop:KingmannforPScocycle}
Let $\mu$ be a    probability measure on the group $\Isom{X}$ satisfying the requirements stated in Theorem \ref{thm:critical exponent for SRS}. Let $\nu$ be a discrete $\mu$-stationary random subgroup of $\Isom{X}$ such that $\nu$-almost every subgroup has  divergence type. Then $\mathrm{P}_\mu$-almost every sample path $\omega \in \Gamma^\mathbb{N}$ and $\nu$-almost every subgroup $\Delta \in \DSub{\Isom{X}}$ satisfy
\begin{equation}
\label{eq:n-roots go to 1}
\lim_{n\to\infty} \pi_\nu(\omega_n,\Delta)^{\frac{1}{n}} = 1.
\end{equation}
\end{prop}

\begin{proof}
To simplify notation write $G = \Isom{X}$.
 For each $s > \delta(\Gamma)$ consider the following double integral
 \begin{equation}
 \label{eq:before fubini}
 \int_{G \times \DSub{G}} \log \frac{\mathcal{P}_{  \Gamma^g}(s)}{\mathcal{P}_\Gamma(s)} \; \mathrm{d} (\mu \otimes \nu)(g,\Gamma).
 \end{equation}
The function in question is $L^1$ by the finite first moment assumption together with Equation (\ref{PSisL1}). We may apply Fubini's theorem and use the fact that $\mu * \nu = \nu$ to evaluate the double integral in Equation (\ref{eq:before fubini}) as being equal to
\begin{equation}
\int_{\mathrm{DSub}(G)} \left( \int_G \log \mathcal{P}_{ \Gamma^g}(s) \; \mathrm{\mu}(g) - \log \mathcal{P}_\Gamma(s) \right) \; \mathrm{d} \nu(\Gamma) = 0.
\end{equation}
 Thus, the dominated convergence theorem combined with the limit in Equation \ref{quasiPoincare} imply that the function $\log \pi_\nu $ is   $(\mu \times \nu)$-integrable. More precisely, it satisfies 
\begin{equation}
\label{eq:bound on integral of log}
\abs*{ \int_{G \times \DSub{G}} \log \pi_\nu (g,\Gamma )\; \mathrm{d}(\mu \otimes \nu)(g,\Gamma)}   \leq C_0
\end{equation}
for some constant $C_0 > 0$ depending only on the metric space $X$. In particular,  the constant $C_0$ is independent   of the measure $\mu$, so that Equation \ref{eq:bound on integral of log} continues to hold, say, if $\mu$ is replaced by any of its convolution powers $\mu^{*n}$.
 
We proceed with establishing Equation (\ref{eq:n-roots go to 1}).  Consider the sequence of functions $f_n:G^\mathbb{N} \times \DSub{G}  \to \mathbb{R}_{>0}$ given by 
\begin{equation}
f_n((g_i),\Delta)=\log \pi_\nu (g_1...g_n,\Delta) \quad \forall (g_i)\in G^\mathbb{N}, \Delta \in \DSub{G}
\end{equation}
and for all indices $n \in \mathbb{N}$. Define the transformation 
\begin{equation}
T:G^\mathbb{N}\times \DSub{G} \to G^\mathbb{N}\times \DSub{G}
\end{equation}
by 
\begin{equation}
T((g_i)^\infty_{i=1},\Delta)=((g_{i+1})^\infty_{i=1},\Delta^{g_1}).
\end{equation}
The transformation $T$ preserves the product probability measure $\mu^\mathbb{N}\times \nu$ and acts ergodically. Since $\pi_\mu$ is a multiplicative quasi-cocycle we get
 \begin{equation}
 f_n \circ T^m +f_m-C_1\leq f_{n+m}\leq  f_n \circ T^m +f_m+C_1
 \end{equation}
 for all $n,m\in\mathbb{N}$ where $C_1>0$ is some constant depending only on the metric space $X$.
Denote  $f'_n=f_n+C_1$. Since $\log \pi_v$ belongs to $L^1(\mu \times \nu)$, so do the functions $f_n$ and $f'_n$ for all $n \in \mathbb{N}$. The sequence of functions $f'_n$ satisfies the inequalities
\begin{equation}
f'_{n+m}\leq  f'_n \circ T^m +f'_m 
\end{equation}
for all $n,m \in \mathbb{N}$. 
Kingman's subadditive ergodic theorem implies that
\begin{equation}
\lim_{n\to \infty}\frac{f_n}{n}=\lim_{n\to \infty}\frac{f'_n}{n}=\lim_{n\to \infty} \int \frac{f'_n}{n} \; \mathrm{d}(\mu^\mathbb{N}\times \nu) = \lim_{n\to \infty} \int \frac{f_n}{n} \; \mathrm{d}(\mu^\mathbb{N}\times \nu)
\end{equation}
in the sense of pointwise convergence  at  $(\mu^\mathbb{N}\times \nu)$-almost all points. Note that 
\begin{equation}
 \abs*{ \int_{G \times \DSub{G}} f_n \; \mathrm{d} (\mu^{*n} \times \nu) }=\abs*{\int_{G \times \DSub{G}} \log \pi_\nu  \; \mathrm{d} (\mu^{*n} \times \nu)  } \leq C_1
\end{equation}
independently of $n$, as we have explained above,
Therefore, we conclude that $\frac{f_n}{n} \to 0$ pointwise at $(\mu^\mathbb{N}\times \nu)$-almost all points as $n \to \infty$. This is equivalent to the desired conclusion in Equation (\ref{eq:n-roots go to 1}).
\end{proof}

\begin{proof}[Proof of the divergence type case in Theorem \ref{thm:critical exponent for SRS}]
Write   $G = \Isom{X}$ to simplify notations.
Let $\nu$ be a discrete divergence type $\mu$-stationary random subgroup of the group $G $.  Consider the multiplicative cocycle $\pi_\nu : G \times \DSub{G} \to \mathbb{R}_{> 0}$ associated to the random subgroup $\nu$.

Assume without loss of generality that the measure $\nu$ is ergodic. Therefore the critical exponent  regarded as a $\nu$-measurable function $\delta : \DSub{G} \to \mathbb{R}_{\ge 0}$   is $\nu$-almost everywhere constant \cite[p. 424]{gekhtman2019critical}. Denote this constant by $\delta(\nu)$. So  $\delta(\Delta) = \delta(\nu)$ for $\nu$-almost every subgroup $\Delta$. In this terminology, our goal is to prove that $\delta(\nu) \ge \delta(\mu)$.

Take a Borel subset $Y \subset \DSub{G}$ with $\nu(Y) > 1 - \zeta(\mu)$ with the properties provided by  \cite[Proposition 7.2]{gekhtman2019critical}. Here $\zeta(\mu) > 0$ is the constant provided by Lemma \ref{lemma:from the blackboard}. For each discrete subgroup $\Delta$ denote 
\begin{equation}
W_\Delta  = \{\gamma  \in \Gamma \: : \: \Delta^\gamma \in Y\}.
\end{equation}
There is some $k \in  \mathbb{N}$ such that $\nu$-almost every discrete subgroup $\Delta \in Y$ satisfies
\begin{equation}
\sum_{\substack{\gamma \in W_\Delta   \\ i-k \le \|\gamma\| \le i}} \pi_\nu(\gamma,\Delta) \le c_1 e^{\delta(\nu) i}    
\end{equation}
for some constant $c_1 > 0$ and for all $i \in \mathbb{N}$, see \cite[Proposition 7.3]{gekhtman2019critical}. We remark that these two propositions from \cite{gekhtman2019critical} were proved in the context of invariant random subgroups but their proofs don't use the invariance of the measure $\nu$ in any way. They are valid for any probability measure  on the space $\DSub{G}$ whatsoever. 
Changing summation from annuli to spheres gives
\begin{equation}
\sum_{\substack{\gamma \in W_\Delta    \\   \|\gamma\| \le j}} \pi_\nu(\gamma,\Delta) \le \sum_{i=1}^j \sum_{\substack{\gamma \in W_\Delta    \\ i-k \le \|\gamma\| \le i}} \pi_\nu(\gamma,\Delta) \le c_1 \sum_{i=1}^j e^{\delta(\nu) i} \leq c_2 e^{\delta(\nu)j}   
\end{equation} 
for all indices $j \in \mathbb{N}$ where $c_2=\frac{c_1}{1-e^{-\delta(\nu)}} $.

We contrast the previous paragraph with the fact that $\lim_n \pi_\nu(\omega_n,\Delta)^\frac{1}{n} =1$ holds true for $\mathrm{P}_\nu$-almost every sample path $\omega$ and $\nu$-almost every discrete subgroup $\Delta$. Fix an arbitrary parameter $a > 0$.   Proposition \ref{prop:TanakawithF} shows that   $\nu$-almost every subgroup $\Delta$ satisfies
\begin{equation}
\label{eq:is it possible}
\limsup_{j \to \infty} e^{-(\delta(\mu)-a)j} \sum_{\substack{\gamma \in W_\Delta \\  \|\gamma\|\leq j}} \pi_\nu(\gamma,\Delta)>0.
\end{equation}
Equation (\ref{eq:is it possible}) is only possible  if $ \delta(\mu) - a \le \delta(\nu)$. Taking the parameter $a$ to be arbitrary small we get $\delta(\mu) \le \delta(\nu)$ as required.
\end{proof}

The following strategy  is taken from \cite{gekhtman2019critical}.

\begin{proof}[Proof of the strict inequality in Theorem \ref{thm:critical exponent for SRS}]
Let  be $\nu$ be a discrete $\mu$-stationary random subgroup on $\Isom{X}$. We have already proved that $\mathcal{P}_\Delta \left(\frac{\delta(\mu)}{2}\right) = \infty $ so that in particular $\delta(\nu) \ge \frac{\delta(\mu)}{2}$. Assume towards contradiction that $\delta(\nu) = \frac{\delta(\mu)}{2}$. This means that the $\mu$-stationary random subgroup $\nu$ is almost surely of divergence type. It follows from the above proof  that $\delta(\nu) \ge \delta(\mu)$. This is a contradiction as $\delta(\mu) > 0$.
\end{proof}

The above paragraphs complete the proof of Theorem \ref{thm:critical exponent for SRS}.

\subsection*{A probability 
measure $\mu$ maximizing $\delta(\mu)$}
The exact lower bound on the critical exponent of discrete stationary random subgroups obtained in Theorem \ref{thm:critical exponent for SRS} depends on the parameter $\delta(\mu) = h(\mu)/l(\mu)$ of the probability measure $\mu$. In this sense, our results are stronger, the greater is the value of $\delta(\mu)$. We now discuss a particular construction  maximizing the parameter $\delta(\mu)$ following Connell and Muchnik  \cite{Connell-Muchnik}.

Recall that $X$ is a proper  Gromov hyperbolic geodesic metric space with a non-elementary group of isometries  $\Isom{X}$ and a fixed arbitrary base point $x_0 \in X$. Assume that $\Isom{X}$ admits a uniform lattice denoted by $\Gamma$. 

Let $\eta^\Gamma$  be a $\Gamma$-quasi-conformal density of dimension $\delta(\Gamma)$. Normalize so that $\eta^{\Gamma}_{x_0}(\partial X) = 1$.
The main result of \cite{Connell-Muchnik} asserts that there is a probability measure $\mu_\Gamma$ supported\footnote{In fact $\mu_\Gamma$ can be taken so that $\mathrm{supp}(\mu_\Gamma) = \Gamma$,   see  \cite[Remark 0.6]{Connell-Muchnik}.} on the lattice $\Gamma$ such that the probability measure $\eta^{\Gamma}_{x_0}$ on the boundary $\partial X$ is $\mu$-stationary. Now, on the one hand, the Hausdorff dimension of a $\mu$-stationary probability measure on the boundary equals $h(\mu)/l(\mu)$ \cite[Theorem 1.1]{tanakadimension}. On the other hand, the Hausdorff dimension of $\delta(\Gamma)$-dimensional $\Gamma$-quasi-conformal measures equals $\delta(\Gamma) = \mathrm{dim}_\mathrm{Haus}(\partial X)$ \cite[Theorem 7.7]{coornaert1993mesures}. We deduce that the measure $\mu_\Gamma$ constructed by the Connell and Muchnik satisfies 
\begin{equation}
\delta(\mu_\Gamma)=\delta(\Gamma)=\mathrm{dim}_\mathrm{Haus}(\partial X). 
\end{equation}
Recall  that $\delta(\Gamma)$ is the maximal possible such value for $\delta(\mu)$ by the   fundamental inequality of Guivarch  \cite{Guivarch}.

In the special case where the space $X$ is assumed to be $
\mathrm{CAT}(-1)$, the measure $\mu_\Gamma$ constructed in \cite{Connell-Muchnik} can be chosen to have finite first moment. The analogous fact for symmetric spaces  of rank one (or more generally, universal covers of closed pinched negatively curved Riemannian manifolds) follows by the method of discretization of Brownian motion, see  \cite{furstenberg1971random,lyons1984function,Ledrappierharmonic,margulis1991discrete,ballmann1992discretization,blachere2011harmonic}. For symmetric spaces, the $\mu_\Gamma$-stationary probability measure on the boundary $\partial X$ coincides with the unique $K$-invariant one, where $K = \mathrm{stab}_{\Isom{X}}(x_0)$.

For general Gromov hyperbolic spaces which are not   $\mathrm{CAT}(-1)$,    it is not known whether one   can find such  a measure $\mu_\Gamma$ supported on the lattice $\Gamma$  with $\delta(\mu_\Gamma)=\delta(\Gamma)$ and with finite first moment. However, the construction of Gouezel--Matheus--Maucourant   \cite[Theorem 1.4]{gouezel2018entropy} exhibits a sequence  of probability measures $\mu_i$ supported on finite subsets of the lattice $\Gamma$ and satisfying $\delta(\mu_i) \to \delta(\Gamma)$. In fact, each such $\mu_i$ can be taken to be the uniform probability measure supported on the   annulus given by
$ \{\gamma \in \Gamma \: : \: k \le \|\gamma\| < k+1\}$.

It is interesting to specialize Theorem \ref{thm:critical exponent for SRS} to the  situation when we know of an explicit   measure $\mu_\Gamma$ satisfying $\delta(\mu_\Gamma) = \dim_\textrm{Haus}(\partial X)$. In that case, we immediately obtain the following.

\begin{cor}
\label{cor:srs for connel-muchnik}
Let $X$ be a $\mathrm{CAT}(-1)$  space. Assume that $\Isom{X}$ admits a uniform lattice $\Gamma$. Fix\footnote{Such a probability measure $\mu_\Gamma$ always exists for any $\mathrm{CAT}(-1)$-space admitting a uniform lattice; see the preceding discussion relying on the work of \cite{Connell-Muchnik}.} a probability measure  $\mu_\Gamma$  whose support generates $\Gamma$ as a semigroup, of finite first moment and satisfying $\delta(\mu_\Gamma) = \dim_\mathrm{Haus}(\partial X)$.  Let $\nu$ be a non-trivial discrete $\mu$-stationary random subgroup  on $\Isom{X}$. Then 
\begin{itemize}
\item $\delta(\Delta) > \frac{1}{2} \dim_\mathrm{Haus}(\partial X)$ for $\nu$-almost every subgroup $\Delta$, and
\item if $\nu$-almost every subgroup $\Delta$ has divergence type then $\delta(\Delta) = \dim_\mathrm{Haus}(\partial X)$.
\end{itemize}
\end{cor}
%\begin{proof}
%The probability measure $\mu_\Gamma$ constructed by Connell and Muchnik   satisfies $\delta(\mu_\Gamma) = \dim_\textrm{Haus}(\partial X)$ and has finite first moment, see \cite{Connell-Muchnik}.  The corollary follows from these facts combined  with the main results of this section.
%\end{proof}

Corollary \ref{cor:srs for connel-muchnik} certainly applies to all rank one symmetric spaces and Bruhat--Tits buildings.
Lastly we consider stationary random subgroups of free groups.

\begin{cor}
Let $F_k$ be the free group on $k$ generators. Let $\mu_k$ be the uniform measure on the standard symmetric generating set. Then any $\mu_k$-stationary random subgroup of $F_k$ has
\begin{itemize}
\item $\delta(H) > \frac{\log(2k-1)}{2}$ for $\nu$-almost every subgroup $\Delta$, and
\item if $\nu$-almost every subgroup $H$ has divergence type then $\delta(H) = \log(2k-1)$.
\end{itemize}
\end{cor}
\begin{proof}
The finitely supported measure $\mu_k$ satisfies
\begin{equation}
\delta(\mu) = \frac{l(\mu)}{h(\mu)} = \dim_\mathrm{Haus}(\partial X) = \log(2k-1)
\end{equation}
  by \cite{ledrappier2001some}. The desired conclusion follows from this and from Theorem \ref{thm:critical exponent for SRS}.
\end{proof}

\section{Random walks and confined subgroups}
\label{sec:random walks and confined subgroups}

In this section, we apply our results on critical exponents of stationary random subgroups %and on non-condensing probability measures
to show that subgroups of small critical exponent are not confined.

\subsection*{Confined subgroups}

Let $G$ be a second countable locally compact group. 

\begin{defn}
\label{def:confined subgroup}
A subgroup $H$ of the group $G$  is called \emph{confined} if there is a compact subset $K \subset G$   so that  $(H^g \cap K) \setminus \{e\} \neq \emptyset $ for every element $g \in G$.
\end{defn}
The notion of confined subgroups admits an equivalent characterization in terms of the Chabauty topology.

\begin{prop}
\label{prop:Chabauty characterization of confined}
If the trivial subgroup $\{e\}$ does not belong to the Chabauty closure  $\overline{  \{H^g \: : \: g \in G\} } $    then the subgroup $H \le G$ is  confined. The converse direction holds provided that the group $G$ has no small subgroups\footnote{A topological group $G$ has \emph{no small subgroups} if it admits a neighborhood of the identity containing no non-trivial subgroups. For example, Lie groups as well as discrete groups have no small subgroups.}. 
\end{prop}
\begin{proof}
Assume that the trivial subgroup is not in the Chabauty closure of the family of all subgroups conjugate to $H$. So, there is some Chabauty neighborhood $\Omega$ with $\{e\} \in \Omega$ but $H^g \notin \Omega$ for all elements $g \in G$. The description of the standard sub-basis for the Chabauty topology given in  \cite[\S1]{gelander2015lecture} shows that $\Omega$ contains a subset of the form $\{ L \in \Sub{G} \: : \: L \cap K = \emptyset\}$ for some compact subset $K \subset G $. The desired conclusion follows.

For the converse direction assume that the group $G$ has no small subgroups  and that $H$ is a confined subgroup. In particular, there is some compact subset $K \subset G$ so that $(H^g \cap K) \setminus \{e\} \neq \emptyset$ for every element $g \in G$. Let $V$ be a symmetric relatively compact identity neighborhood  in the group $G$ containing no non-trivial subgroups. Then the compact subset $K$ can be replaced with $K_1 = (K \cup \overline{V}^2) \setminus V$ to exhibit that the subgroup $H$ is confined. This means that any subgroup in the Chabauty closure $\overline{ \{H^g \: : \: g \in G\} }$ intersects the compact set $K_1$ non-trivially and is in particular non-trivial.
\end{proof}

\begin{example}
Let $G$ be a rank one simple Lie group with associated symmetric space $X$. A discrete torsion-free subgroup $\Gamma$ of the Lie group $G$ is confined if and only if there is an upper bound on the injectivity radius at all points of the locally symmetric space $M_\Gamma = \Gamma \backslash X$.
\end{example}

\subsection*{Continuity of the critical exponent} 
Our method relies on the fact that the critical exponent is   lower semi-continuous with respect to the Chabauty topology.

\begin{prop}
\label{prop:delta is semi-continuous}
Let $\Delta_n \to \Delta$ be a Chabauty converging sequence of discrete subgroups. Assume that the subgroup $\Delta$ is discrete. Then
\begin{equation}
\liminf_n \delta(\Delta_n) \ge \delta(\Delta).
\end{equation}
\end{prop}
This statement is well-known to experts. See for instance \cite[Theorem 7.7]{mcmullen1999hausdorff} for the case of Kleinian groups or \cite[Proposition 7.2]{matsuzaki2020normalizer} for general Gromov hyperbolic spaces (but with an additional assumption that the groups $\Gamma_n$ are of divergence type). As we could not locate a reference for the general statement (without this additional assumption) we provide a proof for the readers' convenience.
\begin{proof}[Proof of Proposition \ref{prop:delta is semi-continuous}]
Fix an arbitrary base point $x_0 \in X$. Denote $\delta_0 = \liminf \delta(\Delta_n)$. 
Up to passing to a subsequence we may assume that $\lim_n \delta(\Delta_n) =   \delta_0$. Let $\mathcal{M}(\partial X)$ denote the space of all positive measures on the boundary $\partial X$. For each $n$ let $\eta_n : X \to \mathcal{M}(\partial X)$ be some $\Delta_n$-quasi-conformal density of exponent $\delta(\Delta_n)$ normalized so that $(\eta_n)_{x_0}(\partial X) = 1$, in other words so that   $(\eta_n)_{x_0}$ is a probability measure. Let $\eta$ be any weak-$*$ accumulation point  of the probability measures $(\eta_n)_{x_0}$ so that $\eta$ is a probability measure on the boundary $\partial X$. Enrich $\eta$ to a quasi-conformal density by setting
\begin{equation}
\eta_x(\zeta) = e^{\delta_0 \beta_X(x_0,x)} \eta(\zeta)  \quad \forall x \in X, \zeta \in \partial X.
\end{equation}
Then $\eta$ satisfies the axioms of a $\Delta$-quasi-conformal density of exponent $\delta_0$. Note that a quasi-conformal measure is not required to have its support coincide with the limit set. Since $\Delta$ admits such a quasi-conformal measure we infer that $\delta(\Delta) \le \delta_0$, see \cite[Corollaire 6.6]{coornaert1993mesures} or \cite[Theorem 2.6]{matsuzaki2020normalizer}.
\end{proof}

\subsection*{Subgroups of small critical exponent are  not confined.}

We consider discrete subgroups of simple rank one analytic groups and show that small critical exponent implies that the subgroup in question is not confined. The  Archimedean and the non-Archimedean cases will be treated in parallel.

In the real analytic case  let $G$ be a rank one simple Lie group with associated symmetric space $X$. Fix an arbitrary base point $x \in X$.

In the non-Archimedean case let $k$ be a non-Archimedean local field. Let $\mathbb{G}$ be a  connected simply-connected semisimple $k$-algebraic linear group with $\mathrm{rank}_k(\mathbb{G}) = 1$ and without $k$-anisotropic factors. Take $G = \mathbb{G}(k)$ so that $G$ is a $k$-analytic group. Let $X$ be the Bruhat--Tits building associated to the group $G$. This means that  $X$ is a $(p^{e_1} + 1,p^{e_2} + 1)$-biregular tree  for some $e_1,e_2 \in \mathbb{N}$ \cite{carbone2001classification} where $p$ is the characteristic of the residue field of $k$.  Fix an arbitrary vertex $x \in X$.

 %\subsection*{Discreteness of limiting stationary random subgroups - Lie group case}
 
In either case the metric space $X$ is Gromov hyperbolic and the analytic group $G$ is acting on the space $X$ by isometries. To simplify the situation we will find it convenient to pass to the quotient by the finite kernel $\mathrm{Z}(G)$ of the action of the analytic group $G$ on the metric space $X$  and   to  assume in effect that $G \le \Isom{X}$.

Denote $K = \mathrm{stab}_G(x)$ so that $K$ is a maximal compact subgroup of the analytic group  $G$.
 The compact subgroup $K$ is acting transitively on the visual boundary $\partial X$ of the Gromov hyperbolic space $X$. 
 %Therefore the boundary  $\partial X$ can be identified with the coset space $K/M$ for some closed subgroup $M$. 
 Let $m$ denote the unique $K$-invariant probability measure\footnote{The probability measure $m$ can be identified with the Haar measure on the coset space $K/M$ where the closed subgroup $M$ is the stabilizer of some boundary point $\infty \in \partial X$.} on the visual boundary $\partial X$.  Note that the measure $m$ depends on the the choice of the point $x \in X$.

Let $\mu_0$ be the explicit bi-$K$-invariant probability measure on the analytic group $G$ constructed in the work \cite{GLM} for the Archimedean as well as for the   non-Archimedean case. 
 %The measure $\mu_0$  is bi-$K$-invariant.
 The measure $\mu_0$ has the following useful property ---  for every discrete group $\Gamma$ of the analytic group $G$, any weak-$*$ accumulation point   of the sequence of probability measures $\frac{1}{n} \sum_{i=1}^n \mu_0^{*i} * \mathrm{D}_{\Gamma}$ is a \emph{discrete} $\mu_0$-stationary random subgroup      \cite[Theorem 1.6]{fraczyk2023infinite}. 
 
We remark that  in the zero characteristic non-Archimedean case there is no need to consider the measure $\mu_0$ at all, since the analytic group $G$ has no small discrete subgroups \cite[Part II, Chapter V.9, Theorem 5]{serre2009lie}.
 
 %This argument relies on the inequality established in \cite{GLM}.

Let $\nu_0$ be the  $\mu_0$-stationary  boundary measure $\nu_0$ on the visual boundary $\partial X$ provided by Theorem \ref{thm:boundary measure}. The measure $\nu_0$ is $K$-invariant. This follows from the equation  $\nu_0 = \mu_0 * \nu_0$ and using the fact that the measure $\mu_0$ is bi-$K$-invariant. Since the visual boundary $\partial X$ admits a \emph{unique} $K$-invariant probability measure $m$ it must be the case that   $\nu_0 = m$.

Consider the probability measure $\mu_1$ on the group $\Isom{X}$ as follows. Recall that we are regarding the analytic group $G$ as a subgroup of $ \Isom{X}$.
\begin{enumerate}
    \item  \emph{The real case.} Any connected non-compact semisimple Lie group admits uniform (as well as non-uniform) lattices \cite{borel1963compact}. Fix an arbitrary uniform lattice $\Gamma$ in the Lie group $G$. Take $\mu_1$ to be either the probability measure $\mu_\Gamma$ constructed by Connell--Muchnik \cite{Connell-Muchnik} or the discretization of Brownian motion on the Lie group $G$, supported on the lattice $\Gamma$ and with respect to the base point  $x$. It is a property of these  classes of measures  (discussed in \S\ref{sec:critical exponents}) that the $\mu_1$-stationary boundary measure $\nu_1$ on the visual boundary $\partial X$ is the unique $K$-invariant one. 
    \item \emph{The non-Archimedean case.} The fundamental domain for the action of the full group of isometries $\Isom{X}$ on the Bruhat--Tits tree $X$ is a single edge. In particular, the group $\Isom{X}$ admits a uniform lattice $\Gamma$ of the form  $\Gamma = A \ast B$ where $A$ and $B$ respectively are finite groups acting  simply transitively on the link of two adjacent vertices $v_A$ and $v_B$ of the tree $X$.
Let $\mu_A$ and $\mu_B$ be the uniform probability measures on the two sets $A\setminus \{e\}$ and $B\setminus \{e\}$ respectively. Let $\mu_1$ be the probability measure supported on the uniform lattice $\Gamma$ given by 
\begin{equation}
\mu = \frac{|A|-1}{|A|-1 + \rho} \mu_A + \frac{|B|-1}{|B|-1 + \rho} \mu_B
\end{equation}
where $\rho > 0 $ is the unique positive real number satisfying
\begin{equation}
\frac{|A|-1}{|A|-1+\rho} + \frac{|B|-1}{|B|-1+\rho} = 1.
\end{equation}
The probability  measure $\mu_1$   satisfies $\delta(\mu_1) = \frac{h(\mu_1)}{l(\mu_1)} = \delta(\Gamma)$. Furthermore the $\mu_1$-stationary boundary stationary measure $\nu_1$ on the visual boundary $\partial X$ is $K$-invariant. See  \cite[Proposition 5.2]{mairesse2007random} for both facts. Alternatively, for a less explicit but more uniform treatment, one may take $\mu_1$ to be the measure $\mu_\Gamma$ constructed in \cite{Connell-Muchnik} with respect to this (or any other) uniform lattice.
\end{enumerate}

In either case $\mu_1$ is a probability measure supported on a uniform lattice $\Gamma$ in the group $\Isom{X}$ and the $\mu_1$-stationary boundary measure $\nu_1$ satisfies  $\nu_1 = m$.

We conclude that the stationary boundary measures $\nu_0$ and $\nu_1$ for the two probability measures $\mu_0$ and $\mu_1$ on the analytic group $G$ coincide. These stationary boundary measures on the visual boundary $\partial X$  provide a topological model for the Poisson boundary. It follows that a given random subgroup $\nu$ of the analytic group $G$    is $\mu_0$-stationary if and only if it is $\mu_1$-stationary (see e.g. the discussion on \cite[p. 428]{GLM}).

\begin{theorem}
\label{thm:not confined both cases at once}
Let $G$ be a  simple analytic group\footnote{Here   $G$ is understood to be either a  simple rank one Lie group or a non-Archimedean analytic group of the form $\mathbb{G}(k)$   considered above.} of rank one. Let $X$ be the  associated symmetric space or Bruhat--Tits tree. Let $\Delta$ be a discrete subgroup of the analytic group $G$. If $\delta(\Delta) \le \frac{\dim_{\mathrm{Haus}}(\partial X)}{2}$ then $\{e\} \in \overline{\{\Delta^g \: : \: g \in G\}}$.
\end{theorem}

Recall that if $G$ is a Lie group then the statement of Theorem  \ref{thm:not confined both cases at once} can be reformulated to say that $\Delta$ is not confined, see Proposition \ref{prop:Chabauty characterization of confined} for details.

\begin{proof}[Proof of Theorem \ref{thm:not confined both cases at once}]
Let $\mu_1$ be the probability measure supported on some fixed uniform lattice in the analytic group $G$ as  constructed above, e.g. the one constructed in \cite{Connell-Muchnik}. It  satisfies $\delta(\mu_1) = \delta(\Gamma) = \frac{\dim_{\mathrm{Haus}}(\partial X)}{2}$.

Consider a discrete subgroup $\Delta$ of the Lie group $G$ and assume that  $\delta(\Delta) \le \frac{\delta(\mu_1)}{2}$. Let $\nu$
be any weak-$*$ accumulation point of the sequence of probability measures $\nu_n = \frac{1}{n} \sum_{i=1}^n \mu_\textrm{0}^{*i} \ast \mathrm{D}_\Delta$. The key property of the probability measure $\mu_0$ implies that the resulting  random subgroup $\nu$ is \emph{discrete} and $\mu_0$-stationary \cite[Theorem 1.6]{fraczyk2023infinite}. The last paragraph of the preceding discussion shows that the random subgroup $\nu$ is at the same time $\mu_1$-stationary.

We remark that in the zero characteristic non-Archimedean case there is no need to consider the measure $\mu_0$ at all and we may work with the measure $\mu_1$ directly. 

We know from Theorem \ref{thm:critical exponent for SRS} that $\nu$-almost every \emph{non-trivial} subgroup $\Lambda$ satisfies $\delta(\Lambda) > \frac{\delta(\mu_1)}{2}$.  One the other hand, the critical exponent   is a conjugation invariant and a lower semi-continuous function on the Chabauty space of discrete subgroups by Proposition \ref{prop:delta is semi-continuous}. So $\delta(\Lambda) \le \delta(\Delta)$ must hold $\nu$-almost surely. We conclude that    $\nu = \mathrm{D}_{\{e\}}$ where $\{e\}$ is  the trivial subgroup of the Lie group $G$. In the particular   the trivial subgroup $\{e\}$   lies is the Chabauty closure of the conjugacy class of the subgroup $\Delta$  by the Portmanteau theorem. 
\end{proof}

\subsection*{General Gromov hyperbolic groups}

In the case of free groups we recover the result of \cite[Corollary 3.2]{fraczyk2020kesten}.

\begin{theorem}
\label{thm:not confined free group case}
Let $F_k$ be the free group of rank $k \in \mathbb{N}$ and $X_k$ be the Cayley graph of $F_k$ with respect to the standard generating set.  Let $H$ be a subgroup of $F_k$ with $\delta(H) \le \frac{\log(2k-1)}{2}$. Then $H$ is not confined.
\end{theorem}
\begin{proof}
The proof follows the same strategy as the proof of Theorem \ref{thm:not confined both cases at once}.  This time we use the uniform probability measure $\mu_k$ supported on the standard symmetric generating set for the group $F_k$. It is shown in \cite{mairesse2007random} that the measure $\mu_k$  satisfies 
\begin{equation}
\delta(\mu_k) =  \delta(\Gamma) = \dim_\mathrm{Haus}(\partial X_k) = \log(2k-1).
\end{equation}
The proof is simpler in this case as there is no need to worry about discreteness of the limiting stationary random subgroup (for the group $F_k$ itself is discrete).
\end{proof}

Finally we state a weaker analogue of Theorems \ref{thm:not confined both cases at once} and \ref{thm:not confined free group case} that applies more generally to all Gromov hyperbolic metric spaces. 
%Its conclusion depends on the specific properties on the measure $\mu$ in question. The proof runs exactly as that of Theorem \ref{thm:not confined Lie group case} with respect to the given  probability measure $\mu$.

%\marginpar{Change: in CAT(-1) can get less or equal (see remark 10.8)}

\begin{theorem}
\label{thm:not confined general Gromov hyperbolic case}
Let $X$ be a proper Gromov hyperbolic geodesic metric space with   isometry group $\Isom{X}$. Assume that the group $\Isom{X}$ admits a discrete subgroup $\Gamma$ which acts properly and cocompactly and    which intersects the elliptic radical $\mathrm{E}(\Isom{X})$ 
trivially. If $\Delta$ is any subgroup of $\Gamma$ with $\delta(\Delta) < \frac{\delta(\Gamma)}{2}$ then $\Delta$ is not confined. 
\end{theorem}

The above result applies in particular in the case where the group $\Isom{X}$ itself is discrete.

\begin{proof}[Proof of Theorem \ref{thm:not confined general Gromov hyperbolic case}]
Let $\mu_k$ be the uniform probability measure supported on the finite set 
\begin{equation}
\label{eq:definition of annulus discrete}
  \{\gamma \in \Gamma \: : \: k \le \|\gamma\| < k+1\}
\end{equation}  
for every $k \in \mathbb{N}$. The sequence of probability measures $\mu_k$ satisfies 
 $\delta(\mu_k)\xrightarrow{k\to\infty}  \delta(\Gamma)$  \cite[Theorem 1.4]{gouezel2018entropy}. In particular  $\delta(\Delta) \le \frac{\delta(\mu_{k_0})}{2}$ holds for some fixed sufficiently large $k_0 \in \mathbb{N}$. Let $\nu$
be any weak-$*$ accumulation point of the sequence of probability measures $\nu_{n} = \frac{1}{n} \sum_{i=1}^n \mu_{k_0}^{*i} \ast \mathrm{D}_\Delta$. Since the convolution powers  $\mu_{k_0}^{*i}$ are all supported on the discrete subgroup $\Gamma$ and as $ \Delta \le \Gamma$, the resulting $\mu_{k_0}$-stationary random subgroup $\nu$ is almost surely contained in the discrete subgroup $\Gamma$. From this point we may conclude exactly as in the proof of Theorem \ref{thm:not confined both cases at once}.
\end{proof}

Lastly, if the proper space $X$ in Theorem \ref{thm:not confined general Gromov hyperbolic case} is  $\mathrm{CAT}(-1)$ rather than simply Gromov hyperbolic, then the strict inequality can be replaced by a non-strict inequality. Namely, any subgroup  $\Delta \le \Gamma$ with $\delta(\Delta)\leq  \frac{\delta(\Gamma)}{2}$ is not confined.
 
%  \begin{remark}
% To replace   the assumption $\delta(\Delta) <\frac{\delta(\Gamma)}{2}$ in Theorem \ref{thm:not confined general Gromov hyperbolic case} by a  non-strict inequality, we require  a probability measure $\mu$ on the group $\Gamma$ whose support generates $\Gamma$ as a semigroup with finite $p$-th moment for some $p>1$  and $\delta(\mu)=\delta(\Gamma)$. When $X$ is a $CAT(-1)$-space, a  measure with finite \emph{first} moment satisfying the other required conditions was constructed by Connell and Muchnik \cite{Connell-Muchnik}. It is unclear to the authors whether this measure has finite $p$-th moment for some $ p > 1$.
%  \end{remark}

\bibliographystyle{alpha}
\bibliography{hyperbolicsrs}

\end{document}